\documentclass[final,onefignum,onetabnum]{siamart171218}

\usepackage{amssymb}
\usepackage{subimages}

\renewcommand{\figdir}{.}

\newtheorem{assumption}{Assumption}
\newtheorem*{remark}{Remark}

\title{On stochastic optimization methods for \\ Monte Carlo least-squares problems\thanks{Submitted on April 26th, 2018.
}}

\author{Gustavo T. Pfeiffer\thanks{Institute of Industrial Science, the University of Tokyo, Tokyo, Japan (\email{pfeiffer@iis.u-tokyo.ac.jp}, \email{ysato@iis.u-tokyo.ac.jp}).} \and Yoichi Sato\footnotemark[2]}

\newcommand{\eq}[1]{\begin{equation}#1\end{equation}}
\newcommand{\topic}[1]{\subsubsection*{#1}}

\newcommand{\smallmat}[1]{\left[\begin{smallmatrix}#1\end{smallmatrix}\right]}

\newcommand\pperp{\protect\mathpalette{\protect\independenT}{\perp}}
\def\independenT#1#2{\mathrel{\rlap{$#1#2$}\mkern2mu{#1#2}}}

\newcommand{\uhat}{\hat}

\newcommand{\lqq}{\begin{displaymath}}
\newcommand{\rqq}{\end{displaymath}}

\begin{document}

\maketitle

\begin{abstract}
This work presents stochastic optimization methods targeted at least-squares problems involving Monte Carlo integration. While the most common approach to solving these problems is to apply stochastic gradient descent (SGD) or similar methods such as AdaGrad~\cite{adagrad} and Adam~\cite{adam}, which involve estimating a stochastic gradient from a small number of Monte Carlo samples computed at each iteration, we show that for this category of problems it is possible to achieve faster asymptotic convergence rates using an \textit{increasing} number of samples per iteration instead, a strategy we call \textit{increasing precision} (IP). We then improve pre-asymptotic convergence by introducing a \textit{hybrid approach} that combines the qualities of increasing precision and otherwise ``constant'' precision, resulting in methods such as the IP-SGD hybrid and IP-AdaGrad hybrid, essentially by modifying their gradient estimators to have an equivalent effect to increasing precision. Finally, we observe that, in some problems, incorporating a \textit{Gauss-Newton preconditioner} to the IP-SGD hybrid method can provide much better convergence than employing a Quasi-Newton approach or covariance-preconditioning as in AdaGrad or Adam.
\end{abstract}


\section{Introduction}
\label{sec:intro}
This work focuses on solving problems of the form
\eq{\text{minimize } f(x) = \frac12||Q(x)||^2\label{eq:objfun}}
where $Q:\mathbb{R}^n\rightarrow \mathbb{R}^m$ cannot be computed exactly in each point $x$, but rather, it can be approximated using pseudorandom numbers --- More specifically, we assume that for any point $x$ of the domain we can compute unbiased estimators $\uhat Q$ and $\uhat J$ for the values of $Q(x)$ and $J(x)=\partial_xQ(x)$, respectively.

For example, that is the case when one has $Q(x) = R(x) - \bar R$, where $R:\mathbb{R}^n\rightarrow \mathbb{R}^m$ is an integral of the form
\eq{R(x) = \int_\Omega L(x,y) dy,\label{eq:omega}}
for some high-dimensional space $\Omega$; while $\bar R \in \mathbb{R}^m$ are observed data. In this case, we can estimate $R(x)$ using Monte Carlo integration, by randomly sampling $\Omega$ following some distribution $\text{pdf}[Y|x]$:
\eq{\uhat R := \frac1N\sum_{i=1}^{N} \frac{L(x,Y_i)}{\text{pdf}[Y_i|x]},\label{eq:montecarlo}}
which will satisfy $E[\uhat R|x] = R(x)$. Meanwhile, an unbiased estimator of the derivative (Jacobian) of $R$ can be computed in the same fashion:
\eq{\widehat{\partial_x R} := \frac1N\sum_{i=1}^{N} \frac{\partial_xL(x,Y_i)}{\text{pdf}[Y_i|x]},\label{eq:jacobian}}
and then we can set $\uhat Q = \uhat R - \bar R$ and $\uhat J = \widehat{\partial_x R}$.

Problems of this kind recurrently appear in areas such as physically-based computer graphics~\cite{pramook,zhao,pfeiffer2012}, molecular physics~\cite{hamiltonian-newton,dna}, nuclear physics~\cite{nuclear} and heat transfer~\cite{heat-transfer}, as inverse problems whose direct counterpart must be solved using Monte Carlo integration.
Throughout this article we refer to this category of problems as \textit{Monte Carlo least-squares}
(MCLS) problems.

Because $f$ and $\nabla f$ cannot be measured exactly in each point, but only estimated using pseudorandom numbers, standard optimization techniques such as Newton's method do not behave well in these problems --- there is no convergence and line search procedures misbehave. Rather, one must use stochastic optimization methods such as the stochastic gradient descent method~\cite{bottou2010}.

In the past decade, stochastic optimization methods have received increasing interest due to their applicability in \textit{large scale learning} (LSL) problems such as training deep neural networks; however, these problems are very different in nature from the MCLS problems we are interested in. Most notably, LSL describes the objective function as a large (but finite) sum of simpler functions, where the noise of its gradient estimator resides in random sampling this very large sum\footnote{In LSL, stochastic optimization solutions model the objective function as a large sum of functions $f(x) = \frac1m\sum_{i=1}^m f_i(x)$ and estimate the gradient typically as $\widehat{\nabla f} = \nabla f_i(x)$ for random  $i$, or as a sum of a few randomly selected $\nabla f_i$ (``mini-batch''). In least-squares problems, that is equivalent to randomly selecting one or a few rows of $Q(x)$ when computing $||Q(x)||^2$, which is useful when the height $m$ of this vector $Q(x)$ is very high.}; while in MCLS, noise originates in generating the $\hat Q$ and $\hat J$ estimators described in the beginning of this section. Thus, MCLS has particularities that do not apply to LSL, and conversely, many of the recently proposed methods for LSL are not applicable to MCLS (e.g.~\cite{sag,fixed1,sohl-dickstein}).

The fundamental method in derivative-based stochastic optimization, stochastic gradient descent (SGD), takes an update rule of the form:
\lqq x_{k+1} = x_k - A_k \widehat{\nabla f}(x_k) \rqq
where $\widehat{\nabla f}(x)$ is an unbiased estimator of $\nabla f(x)$, and $A_k \in \mathbb{R}^{n\times n}$ is a predefined sequence of matrices, typically in the form $A_k = a_k D$, where $a_k\in\mathbb{R}$ is the \textit{step size sequence} (also called \textit{learning rates}), and $D\in\mathbb{R}^{n\times n}$ is the \textit{preconditioner}, usually $D=I$. It has been proven~\cite{efficiency} that when $A_k\sim S^{-1}/k$, where $S=\nabla^2 f(x^*)$ and $x^*$ is the global minimum point, SGD is an asymptotically efficient method (i.e. has optimal convergence rate) with $E[||x_k-x^*||^2] \sim \frac{\text{tr}\{S^{-1}\Sigma^2S^{-1}\}}{k}$ as $k\rightarrow\infty$, where $\Sigma^2 = \text{Var}[\widehat{\nabla f}(x^*)]$. Worth noting, this is only possible when $S$ is known, as it is necessary to set $A_k\sim S^{-1}/k$. We refer to this configuration of SGD as \textit{Hessian-preconditioned SGD}, as $D$ is set to the inverse Hessian matrix. By asymptotically efficient, it means that no other method could possibly have a better convergence rate, assuming the input are the stochastic gradients measured at each point $x_k$.

However, in MCLS, we assume we are given not the gradient samples $\widehat{\nabla f}(x_k)$ at each point $x_k$, but the residual and Jacobian estimators $\uhat Q$ and $\uhat J$, respectively, that are used to estimate the gradient (as in Section~\ref{sec:grad}), which intrinsically contain more information than only the gradient. This implies that we may actually obtain a better asymptotic performance\footnote{By asymptotic performance, we mean the asymptotic decay of the expected square error $E[||x_k-x^*||^2]$ with respect to the total computation time/cost.} than Hessian-preconditioned SGD. We show that under a few assumptions regarding sampling cost (Section~\ref{sec:assumptions}), performance can be improved by gradually increasing the number of samples $(\uhat Q,\uhat J)$ used to compute the gradient at each iteration, a strategy we call \textit{increasing precision} (IP) (Section~\ref{sec:ip}). From IP we then derive a \textit{hybrid approach} (Section~\ref{sec:hybrid}), which combines the qualities of increasing precision (IP) and ``constant'' precision as in SGD. Although the hybrid approach is not as well understood in theory, the resulting methods such as the IP-SGD hybrid and the IP-AdaGrad hybrid methods perform remarkably well in practice. Finally, we observe by experimental analysis how incorporating a \textit{Gauss-Newton preconditioner} to the IP-SGD hybrid can be highly beneficial on MCLS (Section~\ref{sec:gn}), in comparison to existing Quasi-Newton approaches or covariance preconditioning as in AdaGrad~\cite{adagrad} or Adam~\cite{adam}. The theoretic convergence analysis of the proposed methods is presented in a separate section (Section~\ref{sec:analysis}), which proves the convergence speed of IP and a Ruppert-Polyak averaged version of it for the strongly convex case, and provides a more limited theoretical support for the hybrid approach and the stochastic Gauss-Newton methods. Numerical experiments are presented throughout the article (Sections \ref{sec:ip-exp}, \ref{sec:hybrid-exp} and \ref{sec:gn-exp}) whenever new methods are introduced.

\subsection{Related work}
\label{sec:related}

\topic{The use of increasing precision}
The idea of gradually increasing the number of samples used in computing the gradient is not new, as was used for example in \cite{SA-GN, nuclear}. However, our work is original in analyzing how the properties of MCLS problems affect asymptotic behavior when this kind of method is used. The incorporation of Ruppert-Polyak averaging to IP, and the hybrid approach we propose are also assumed to be novel.

\topic{Methods for MCLS} While the application of stochastic optimization to MCLS is recurrent in the literature~\cite{pramook,zhao,pfeiffer2012,hamiltonian-newton,dna}, most works simply apply SGD or some heuristic method, without a proper theoretical analysis of its convergence. We believe this is the first time that methods exploring the particularities of MCLS are proposed.

\topic{Stochastic Gauss-Newton methods} Gauss-Newton approaches to stochastic optimization are rare in the literature. In 1985, Ruppert~\cite{SA-GN} proposed a method that resembles Gauss-Newton to solve systems of equations (that is, with $m=n$ on Equation~\ref{eq:objfun}). A recent method~\cite{kalman-GN} was proposed for the linear case (i.e. linear regression), but is specific to LSL. A few other approaches of stochastic Gauss-Newton and Newton-Raphson have been proposed for particular Monte Carlo applications~\cite{hamiltonian-newton,pfeiffer2012}, with however little theoretical support.

Stochastic variants of Quasi-Newton, on the other hand, are far more common~\cite{wang2017, bordes2009, schraudolph2007, Wei}, although applying Quasi-Newton in a stochastic optimization context faces many complications that Gauss-Newton does not, as will be discussed later (Section~\ref{sec:gn}).

\section{Preliminaries}
In this section, we introduce our notation, assumptions, and provide a brief explanation of how SGD behaves on MCLS problems. We also review the concept of Ruppert-Polyak averaging, which will be incorporated to our methods later on.

\subsection{Remarks on notation}
We do not use a particular typesetting for vectors and random variables; although matrices are always capitalized and a ``$\uhat{\text{ }}$'' symbol always indicates an unbiased estimator, as in $\widehat{\nabla f}(x)$. ``$\widehat{\nabla f}(x)$'' per se is an abuse of notation to denote that the estimator $\widehat{\nabla f}$ is calculated in function of $x$ and satisfies $E[\widehat{\nabla f}|x] = \nabla f(x)$. ``$\pperp$'' denotes independence, and independent identically distributed (i.i.d.) variables are usually denoted with a ``$\text{ }^\prime$'' symbol, as in $X\pperp X^\prime$, or with ``$\text{ }^{(i)}$'', as in $X^{(1)} \pperp X^{(2)} \pperp X^{(3)}$. $\text{Var}[X]$, when $X$ is a vector, indicates covariance matrix: $\text{Var}[X] = E[XX^T]-E[X]E[X]^T$.

\newcommand{\Lim}[1]{\raisebox{0.5ex}{\scalebox{0.8}{$\displaystyle \lim_{#1}\;$}}}
\newcommand{\Limsup}[1]{\raisebox{0.5ex}{\scalebox{0.8}{$\displaystyle \limsup_{#1}\;$}}}
\newcommand{\Liminf}[1]{\raisebox{0.5ex}{\scalebox{0.8}{$\displaystyle \liminf_{#1}\;$}}}
We make extensive use of Bachmann-Landau symbols $o(\cdot)$, $O(\cdot)$, $\Theta(\cdot)$, $\Omega(\cdot)$, $\omega(\cdot)$, and $\sim$ to describe asymptotic behavior, with the following meanings:
\begin{center}
\begin{tabular}{c c c}
 \\ 
$f_k = o(g_k)$ & $\Leftrightarrow$ & $\Lim{k\rightarrow\infty} \frac{|f_k|}{|g_k|} = 0$ \\ 
$f_k = O(g_k)$ & $\Leftrightarrow$ & $\Limsup{k\rightarrow\infty} \frac{|f_k|}{|g_k|} < +\infty$ \\ 
$f_k = \Theta(g_k)$ & $\Leftrightarrow$ & $0 < \Lim{k\rightarrow\infty} \frac{|f_k|}{|g_k|} < +\infty$ \\ 
$f_k = \Omega(g_k)$ & $\Leftrightarrow$ & $\Liminf{k\rightarrow\infty} \frac{|f_k|}{|g_k|} > 0$ \\ 
$f_k = \omega(g_k)$ & $\Leftrightarrow$ & $\Lim{k\rightarrow\infty} \frac{|f_k|}{|g_k|} = +\infty$ \\ 
$f_k \sim g_k$ & $\Leftrightarrow$ & $\Lim{k\rightarrow\infty}\frac{f_k}{g_k} = 1$ \\
 \\ 
\end{tabular}
\end{center}

The subscript of the limits above may be different from ``$k\rightarrow\infty$'' (e.g. ``$x\rightarrow0$'') if specified by context. Whenever applied to random variables, the expressions are meant to hold with probability one (i.e., \textit{almost surely}).

\subsection{Assumptions}
\label{sec:assumptions}

Our methods assume we are able to, for any given $x$, generate a pair $(\uhat Q, \uhat J)$ that unbiasedly estimate $Q(x)$ (as defined in Equation~\ref{eq:objfun}) and $J(x) = \partial_xQ(x)$, respectively. $\uhat Q$ and $\uhat J$ from a same pair $(\uhat Q, \uhat J)$ may be correlated. For simplicity, we assume that the total computational cost is the number of pairs $(\uhat Q, \uhat J)$ computed.

This assumption implies that most of the cost resides in generating $(\uhat Q, \uhat J)$, and that generating $(\uhat Q, \uhat J)$ is not much more costly than generating $\uhat Q$ alone. This is often valid in MCLS applications (e.g. \cite{pramook,pfeiffer2012}) as $\uhat Q$ and $\uhat J$ share much of the computation. In Section~\ref{sec:limitations}, we discuss how our methods perform when these assumptions do not hold.

\subsection{Estimating the gradient}
\label{sec:grad}

It is possible to construct an unbiased estimate of the gradient $\nabla f(x) = J(x)^TQ(x)$ by taking two independent samples $(\uhat Q, \uhat J), (\uhat Q^\prime, \uhat J^\prime)$ and doing
\lqq \widehat{\nabla f} := \uhat J^T\uhat Q^\prime, \rqq
which is unbiased since $\uhat J \pperp \uhat Q^\prime$ and therefore $E[\uhat J^T \uhat Q^\prime] = E[\uhat J^T]E[\uhat Q^\prime] = \nabla f$.

A better choice of gradient estimator in this scenario is
\eq{\widehat{\nabla f} := \frac{\uhat J^T\uhat Q^\prime + \left.\uhat J^\prime\right.^T\uhat Q}2,\label{eq:gradprec2}}
which has lower variance than the previous one (see Lemma~\ref{lemma:var1} in the appendices).

Similarly, when $N$ i.i.d. pairs $(\uhat Q^{(1)}, \uhat J^{(1)}), (\uhat Q^{(2)}, \uhat J^{(2)}), ..., (\uhat Q^{(N)}, \uhat J^{(N)})$ are available, the most appropriate unbiased estimator is
\eq{\widehat{\nabla f} := \frac{1}{N(N-1)}\sum_{1\leq i\neq j \leq N} \left.\uhat J^{(i)}\right.^T\uhat Q^{(j)},\label{eq:gradprec}}
whose variance takes the form
\eq{\text{Var}[\widehat{\nabla f}] = \frac{\Sigma_A^2}N + \frac{\Sigma_B^2}{N(N-1)}\label{eq:var},}
where $\Sigma_A^2$ and $\Sigma_B^2$ are positive semidefinite matrices (see Lemma~\ref{lemma:var2} in the appendices for a proof).

\subsection{The performance of SGD in MCLS}
\label{sec:sgd}
The stochastic gradient descent (SGD) method may be applied to MCLS problems by computing, at each iteration, $N\geq 2$ i.i.d. pairs $(\uhat Q^{(i)}(x_k),\uhat J^{(i)}(x_k))$ and update $x_k$ as:
\eq{x_{k+1} = x_k - A_k\widehat{\nabla f}_N(x_k), \label{eq:sgd}}
where $\widehat{\nabla f}_N$ is the gradient estimator $\widehat{\nabla f}$ of Equation~\ref{eq:gradprec} computed from $N$ samples. 
As discussed in the introduction, Hessian-preconditioned SGD (i.e. SGD with $A_k\sim S^{-1}/k$, where $S=\nabla^2f(x^*)$ and $x^* = \arg\min_x f(x)$) is known to have a convergence rate of $E[||x_k-x^*||^2] \sim \frac{\text{tr}\{S^{-1}\text{Var}[\widehat{\nabla f}(x^*)]S^{-1}\}}k$ as $k\rightarrow\infty$. However, by Equation~\ref{eq:var}, $\text{Var}[\widehat{\nabla f}(x^*)]$ is of the form $\frac{\Sigma_A^2}N + \frac{\Sigma_B^2}{N(N-1)}$, implying that when we write performance in function of the total number of samples computed $t=Nk$, we obtain:
\eq{ E[||x_k-x^*||^2] \sim \frac{\text{tr}\{S^{-1}\text{Var}[\widehat{\nabla f}(x^*)]S^{-1}\}}{t/N} = \frac{\text{tr}\{S^{-1}(\Sigma_A^2+\frac{\Sigma_B^2}{N-1})S^{-1}\}}t, \label{eq:perf-sgd}}
implying asymptotic performance changes according to $N$, being fastest as $N\rightarrow\infty$.

\subsection{Ruppert-Polyak averaging}
A difficulty in applying SGD is the fact that it requires the knowledge of the Hessian $S$ in order to optimally set the sequence $A_k$. A workaround for that is to employ Ruppert-Polyak averaging, which consists of taking asymptotically longer steps than ideal, but returning the average of all iterates:
\lqq \tilde x_{k+1} = \tilde x_k - A_k\widehat{\nabla f}_N(\tilde x_k), \text{ }\text{ }\text{ } x_{k+1} = \frac{\sum_{i=1}^k \tilde x_{i+1}}{k}, \rqq
where $A_k \sim D/k^\alpha$, with $\frac12 < \alpha < 1$ (instead of the $A_k = \Theta(1/k)$ step size decay from Hessian-preconditioned SGD), and $D \in \mathbb{R}^{n\times n}$ symmetric positive definite.

Ruppert~\cite{avg2} and Polyak and Juditsky~\cite{avg1} showed that this update rule provides the same asymptotic behavior as Hessian-preconditioned SGD regardless of the choice of $\alpha$ and $D$, thus not requiring knowledge of the Hessian matrix. In spite of this remarkable theoretical guarantee, Ruppert-Polyak averaging has been reported to take long to reach its asymptotic behavior in practice~\cite{SPSA,bottou2010,xu-avg}, which may suggest the choice of $D$ might heavily impact pre-asymptotic performance.

We will later discuss how the concept of Ruppert-Polyak averaging can be incorporated to our methods.

\section{The increasing precision method}
\label{sec:ip}

The observation that SGD converges faster as $N$ approaches infinity motivates the idea of replacing the constant $N$ of Equation~\ref{eq:sgd} with a non-decreasing sequence $N_k$:
\eq{x_{k+1} = x_k - A_k\widehat{\nabla f}_{N_k}(x_k), \label{eq:ip}}
which we call the \textit{increasing precision} method (IP). Thus, IP may be thought of as a generalization of SGD, where in SGD $N_k$ is constant with respect to $k$.

While in SGD the optimal choice for $A_k$ is $A_k \sim S^{-1}/k$, in IP (\textit{Hessian-preconditioned IP}) it is $A_k \sim N_kS^{-1}/\sum_{i=1}^k N_i$. In Section~\ref{sec:theorems}, we prove that this is true when $N_k$ grows in polynomial rate, provided that $f$ is strongly convex and the distribution of $\widehat{\nabla f}$ has a sufficiently high number of finite moments (Theorem~\ref{thm:ip}). In both Hessian-preconditioned SGD and IP, the convergence rate is
\lqq E[||x_k-x^*||^2] \sim \frac{\text{tr}\left\{S^{-1}\left(\lim_{j\rightarrow\infty}N_j\text{Var}\left[\widehat{\nabla f}_{N_j}\middle| x_j\right]\right)S^{-1}\right\}}{t_k}, \text{ }\text{ (as } k\rightarrow\infty\text{)}\rqq
where $t_k = \sum_{j=1}^k N_j$. Thus, while in SGD the limit term in the expression above converges to $\Sigma_A^2+\frac{\Sigma_B^2}{N-1}$, in IP it converges to $\Sigma_A^2$, implying Hessian-preconditioned IP outperforms any possible configuration of Hessian-preconditioned SGD, since $\Sigma_B^2$ is positive semidefinite.

IP may also be implemented with Ruppert-Polyak averaging, by returning a weighted average of all iterations:
\lqq \tilde x_{k+1} = \tilde x_k - A_k\widehat{\nabla f}_{N_k}(\tilde x_k), \text{ }\text{ }\text{ }  x_{k+1} = \frac{\sum_{i=1}^k N_i \tilde x_{i+1}}{\sum_{i=1}^k N_i}, \rqq
where $N_k = \Theta(k^q)$, for some $q>0$, and $A_k = \frac D{(k+c)^\alpha}$, with $c\geq 0$, $D\in\mathbb{R}^{n\times n}$ symmetric positive definite, and $\alpha$ satisfying $\max\left\{0,\frac{1-q}2\right\}<\alpha<1$.

Just as Ruppert-Polyak averaged SGD (aSGD) has the same asymptotic performance as Hessian-preconditioned SGD, this averaged IP (aIP) method has the same asymptotic performance as Hessian-preconditioned IP, therefore outperforming aSGD. See Theorem~\ref{thm:ipa} in Section~\ref{sec:theorems} for a proof and the detailed conditions in which this property holds.

\subsection{Experiments}
\label{sec:ip-exp}

\tsubimages[]{Comparison of SGD and IP on two different problems (see definition in Appendix~\ref{sec:problems}). The graph is a log-log plot of the average squared error ($E[||x_k-x^*||^2]$) against the total computational cost $t_k = \sum_{i=1}^kN_k$. For each configuration, the darker line in the middle shows the average of 1000 independent optimization runs, while the translucent area around the line shows an error margin of 3 standard deviations of the distribution of the mean.}{n1}{
  \begin{center}
    \subimage[Problem \#1.]{.45}{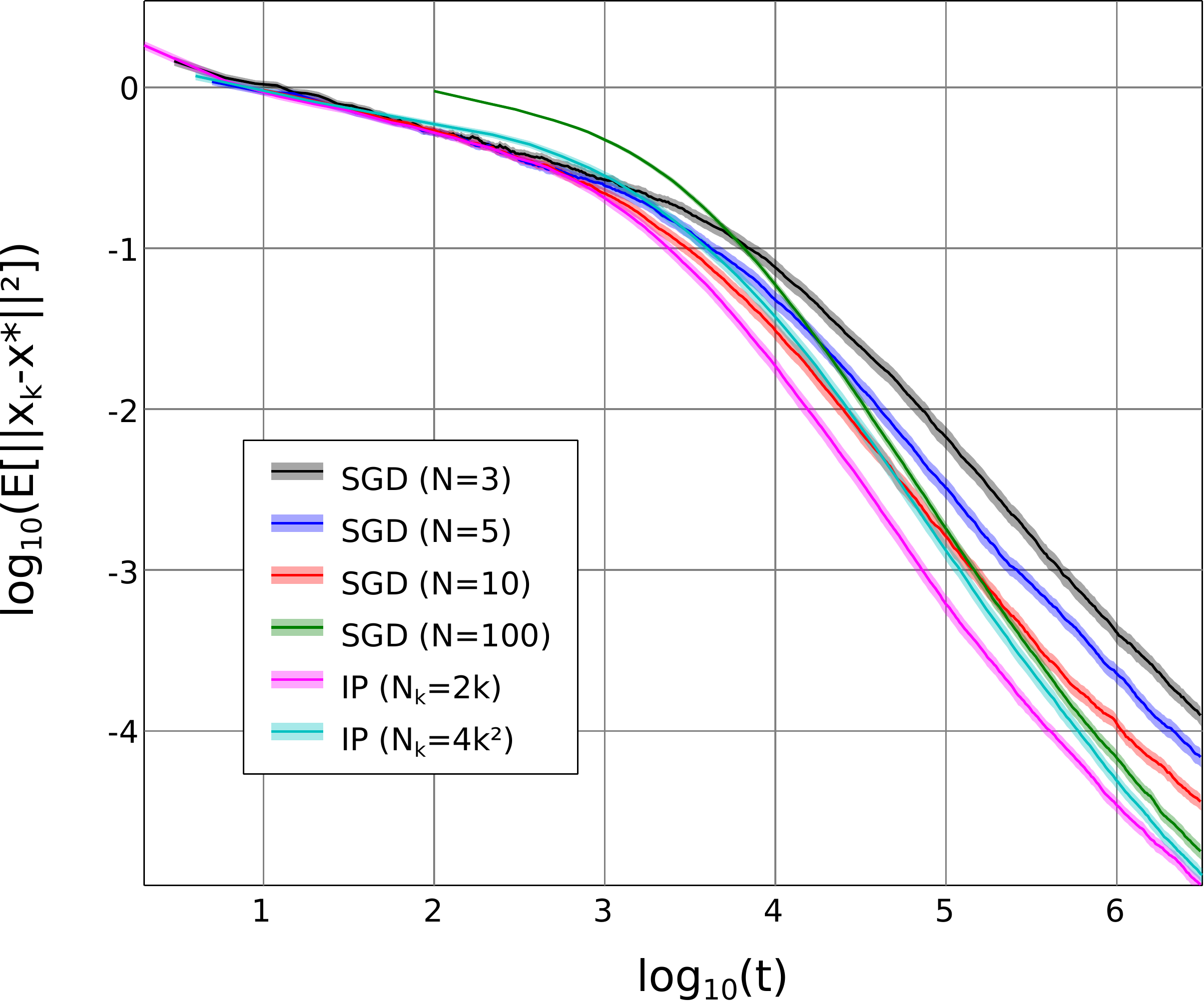}
    \subimage[Problem \#2, $n=10,m=20$.]{.45}{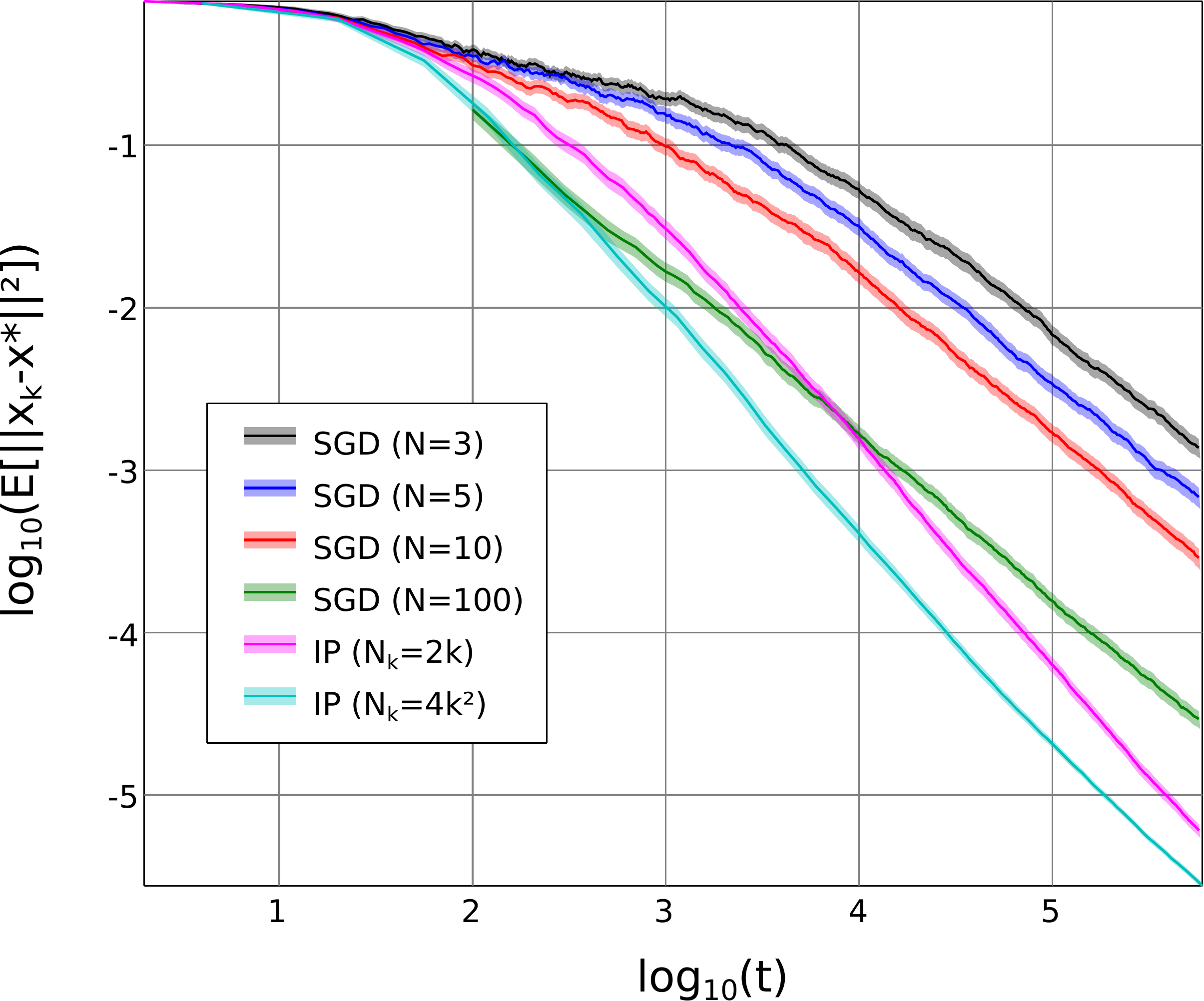}
  \end{center}
}

We compare SGD and IP on two simple problems, described in Appendix~\ref{sec:problems}, by taking the average value of $||x_k-x^*||^2$ among 1000 independent optimization runs. Both methods were preconditioned with a Gauss-Newton approximation of the Hessian at $x^*$. On Figure~\ref{fig:n1}, we observe that SGD converges to different asymptotes depending on $N$, with higher $N$ converging to lower asymptotes, while IP converges to a lower asymptote corresponding to $N=\infty$.

Regarding the pre-asymptotic phase, Figures \ref{fig:n1}(a) and \ref{fig:n1}(b) show different behaviors. On (a), we observe that if $N$ is as high as 100 on SGD, the pre-asymptotic phase slows down, while on (b) this is not observed. Meanwhile, IP is able to provide good performance on both the asymptotic and pre-asymptotic domains, although pre-asymptotic behavior appears to be sensitive to the choice of $N_k$.

\section{The hybrid approach}
\label{sec:hybrid}

As seen in Section~\ref{sec:ip-exp}, although IP provides good theoretical guarantees regarding asymptotic behavior, its pre-asymptotic phase can be very sensitive to the choice of the sequence of number of samples $N_k$. Small $N$ is usually better in the beginning of the optimization, when $x_k$ is far from the optimum $x^*$, so that $x_k$ can quickly approach the neighborhood of $x^*$; while as $x_k$ becomes closer to $x^*$, higher $N$ starts becoming more advantageous due to the reduction of the $\Sigma_B^2$ term of the gradient variance. However, it is difficult to determine this changing point. Would it be possible to simultaneously implement the advantages of both small $N$ and high $N$?

This motivates the \textit{hybrid approach}, which uses only $N=1$ sample per iteration, but combines the observations of the current and previous iterations in order to achieve the same effect of increasing precision of eliminating the $\Sigma_B^2$ term. The hybrid approach consists of replacing the unbiased gradient estimator of Equation~\ref{eq:gradprec} with the following biased estimator:
\eq{g_t = \frac{\sum_{i=1}^{t-1}q_i \left(\uhat J_i^T\uhat Q_t + \uhat J_t^T\uhat Q_i\right)}{2 \sum_{i=1}^{t-1} q_i}\label{eq:hybrid1}}
where $\uhat J_t = \uhat J(x_t), \uhat Q_t = \uhat Q(x_t)$ (i.e., using only $N=1$ sample per iteration), and $q_i$ is a predefined increasing sequence of positive numbers. In addition, on the first iteration ($t=1$), we force $g_t=0$ in order to avoid a division by zero. When we apply this gradient estimator to SGD, i.e., the update rule 
\eq{x_{t+1} = x_t - A_tg_t, \label{eq:ip-sgd-hybrid}}
with $g_t$ following Equation~\ref{eq:hybrid1}, we call this the \textit{IP-SGD hybrid} method. For an explanation of how this connects to the IP method, see Appendix~\ref{sec:hybconn}.

As this approach is essentially only changing the gradient estimator of the method, it can be also applied without difficulties to any stochastic optimization method that takes as input one gradient measurement at each iteration, such as AdaGrad~\cite{adagrad}, Adam~\cite{adam}, or averaged SGD, resulting in what we call, respectively, the IP-AdaGrad hybrid, the IP-Adam hybrid, and the averaged IP-SGD hybrid. However, the optimal choice of $q_i$ may differ for each method, as will be discussed later in this section.

Implementation-wise, it may be numerically more stable, particularly when $q_i$ grows very fast (e.g. superpolynomial growth), to rewrite Equation~\ref{eq:hybrid1} using \textit{forget rates} $\zeta_i = \frac{q_i}{\sum_{j=1}^i q_j}$:
\eq{g_t = \frac{\bar J_t\uhat Q_t + \uhat J_t^T\bar Q_t}2, \text{ }\text{ }\text{ }\text{ }\text{ }\text{ } \left\{\begin{matrix}\bar J_{t+1} = (1-\zeta_t)\bar J_t + \zeta_t \uhat J_t\\\bar Q_{t+1} = (1-\zeta_t)\bar Q_t + \zeta_t \uhat Q_t\end{matrix}\right.,\label{eq:fuzzy}}
where $\zeta_1=1$, $0 < \zeta_t < 1$ (for $t>1$), and the constraint $q_{t+1}>q_t$ implies $\zeta_{t+1}^{-1} < \zeta_t^{-1} + 1$. Also, $\bar Q_1=0$ and $\bar J_1=0$.

Since the hybrid approach does not use an unbiased gradient estimator, the convergence guarantees we had for IP and SGD are not valid anymore, and providing a theoretic analysis of the hybrid method can be challenging. On Section~\ref{sec:asym-hybrid} we analyze a simplified case, which provides some constraints on the choice of forget rates $\zeta_t$ for the IP-SGD hybrid and its averaged counterpart. For the IP-SGD hybrid, any forget rate satisfying $\zeta_t=o(1)$ and $\zeta_t=\omega(1/t)$ should provide good convergence, while for the averaged IP-SGD hybrid, it suffices that $\zeta_t=o(1)$, with a generally faster pre-asymptotic phase when $\zeta_t = \Theta(1/t)$. The optimal choice of $A_t$ is the same as in the original SGD and aSGD methods. We refrained from providing analyses for the IP-AdaGrad and IP-Adam hybrids.

Although not well understood in theory, the hybrid approach performs remarkably well in practice, with the IP-SGD hybrid exhibiting better performance than both IP and SGD, for example, as will be shown in the following subsection.

\subsection{Experiments}
\label{sec:hybrid-exp}

\tsubimages[]{Comparison of forget rates for the IP-SGD hybrid method and its averaged version on Problem \#1, similarly to Figure~\ref{fig:n1}. As we may observe, the forget rates do not significantly impact performance, as long as $\zeta_t = o(1)$. All images show the average of 1000 optimization runs. All cases were preconditioned by a Gauss-Newton approximation of the Hessian at $x^*$, and with $\eta=1$.}{n3}{
  \begin{center}
    \subimage[IP-SGD hybrid, $\zeta_t = \frac{C}{1-(1-C)^t} \sim C$.]{.45}{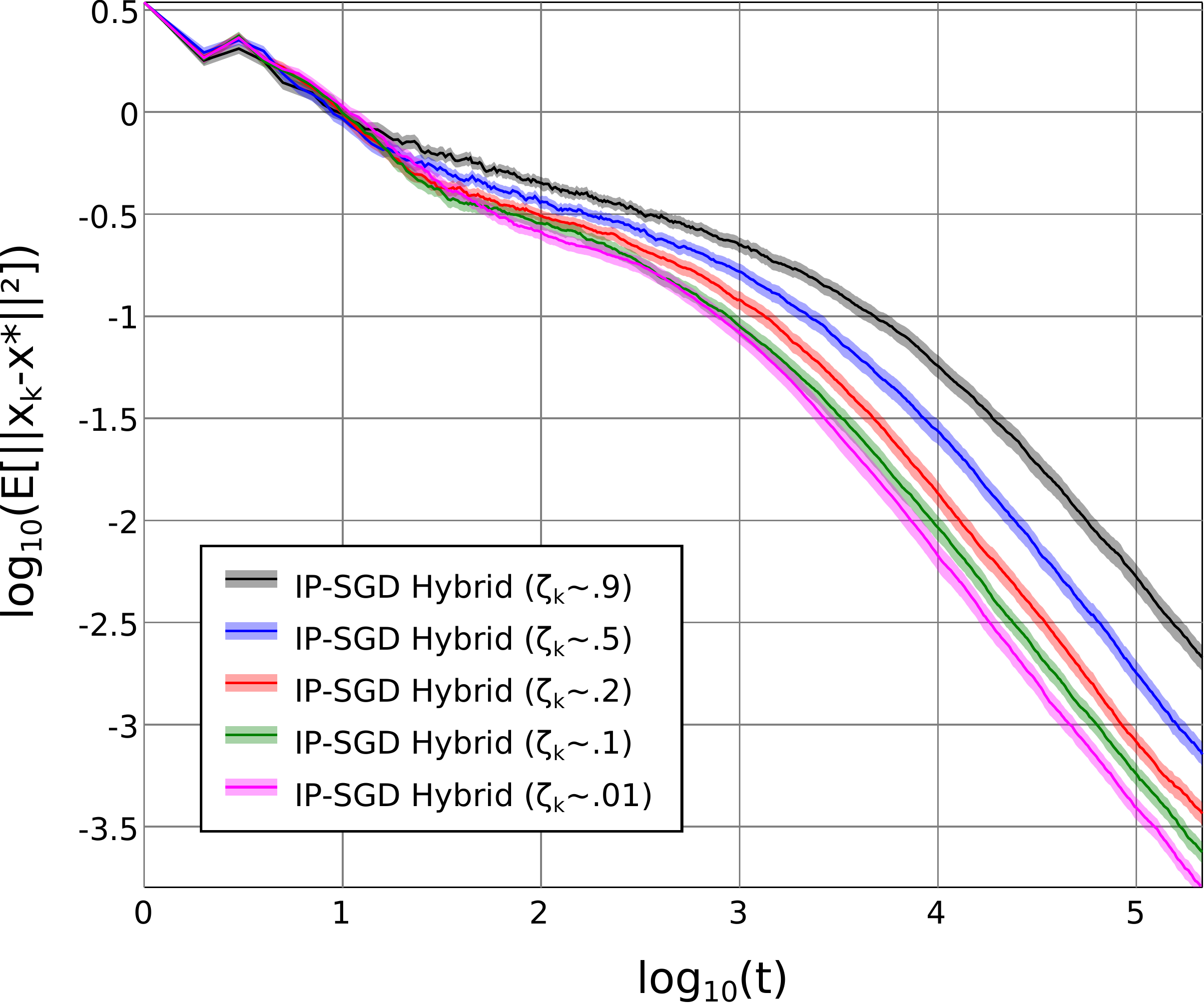}
    \subimage[IP-SGD hybrid, $\zeta_t=\Theta(1/t^a)$, $0<a<1$.]{.45}{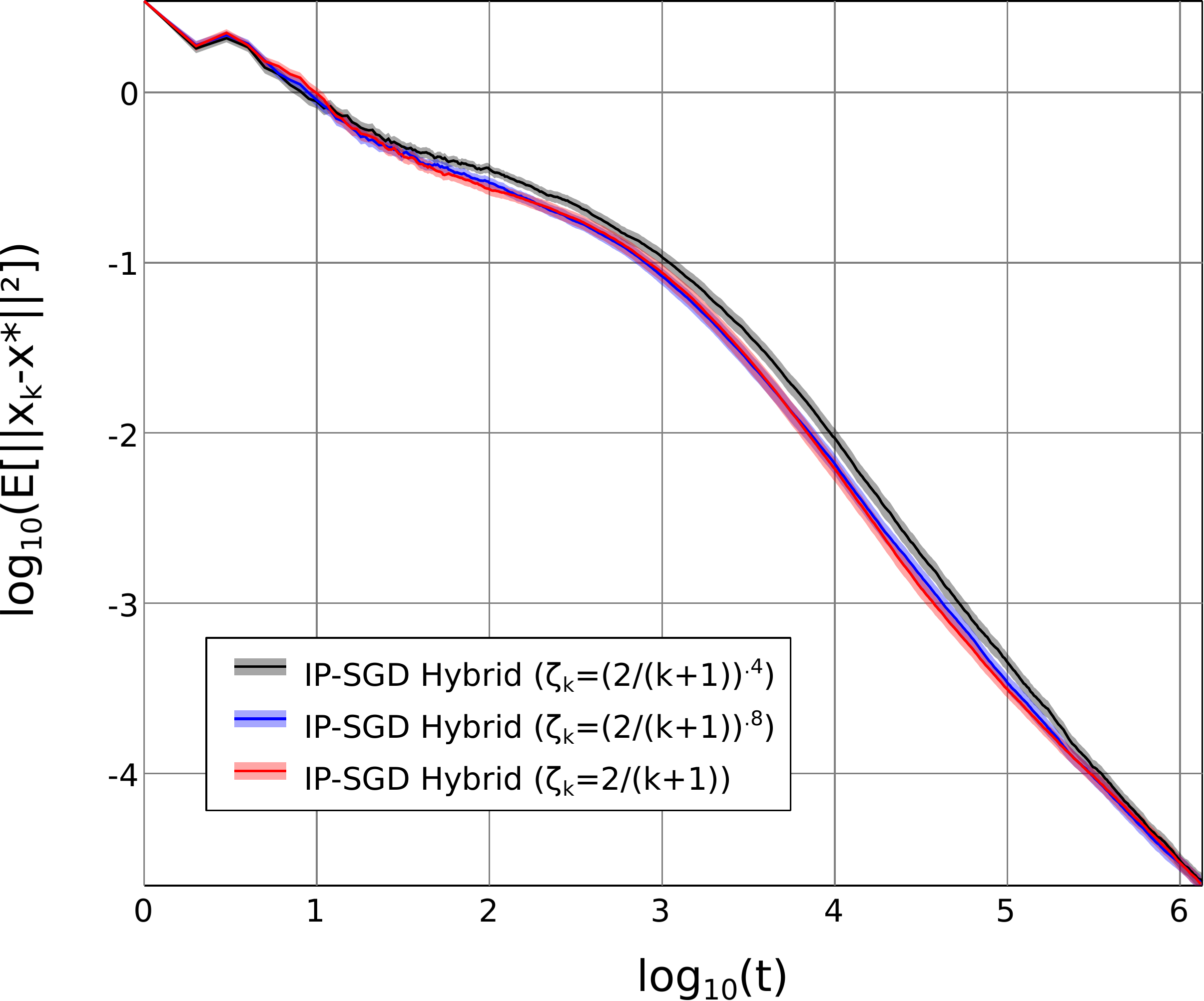}
  \end{center}
  \begin{center}
    \subimage[IP-SGD hybrid, $\zeta_t=\Theta(1/t)$.]{.45}{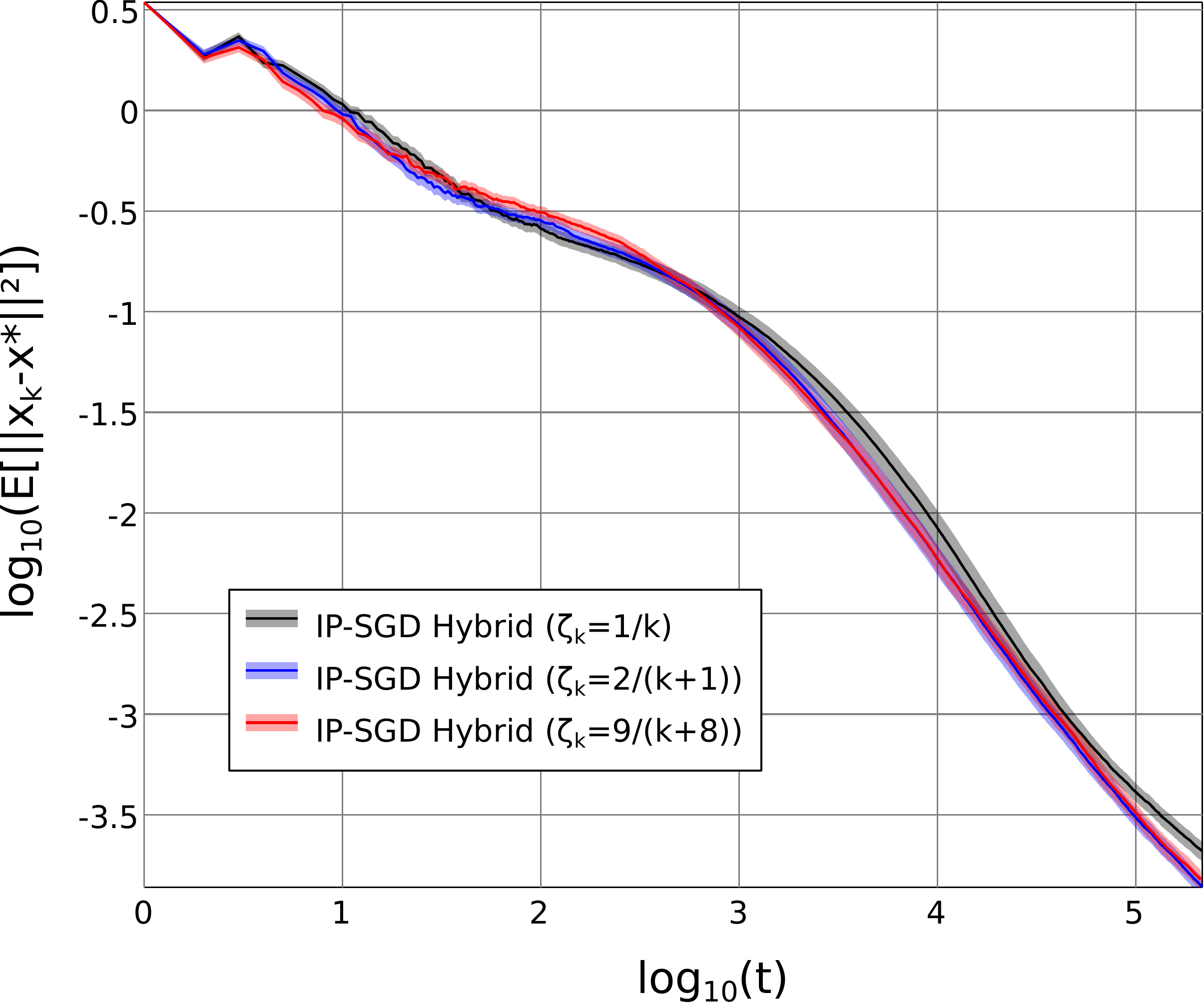}
    \subimage[Averaged IP-SGD hybrid, $\zeta_t=\Theta(1/t)$.]{.45}{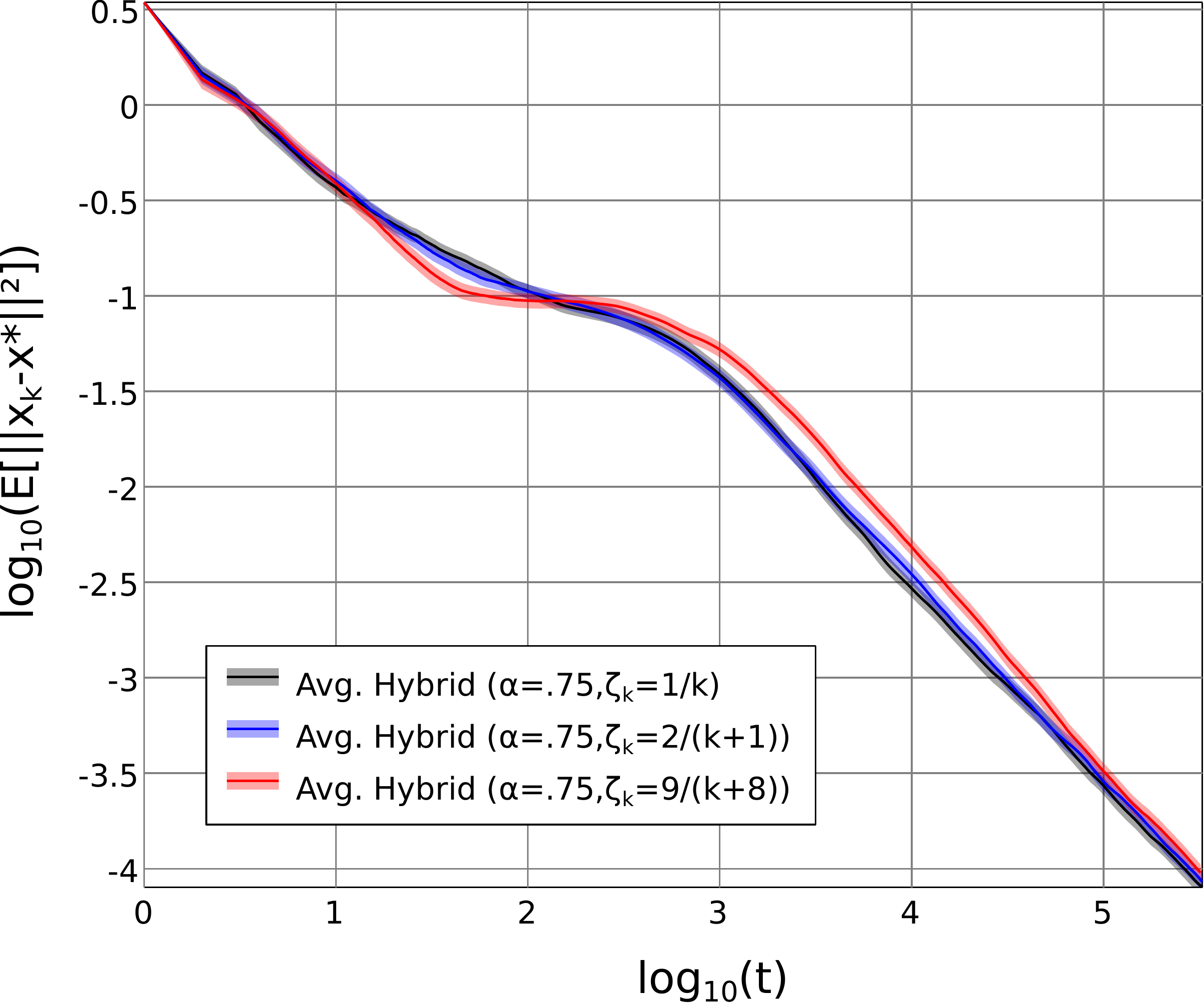}
  \end{center}
}

First, on Figure~\ref{fig:n3}, we assess the impact of the choice of forget rate on the IP-SGD hybrid method and its averaged version. Apparently, as long as $\zeta_t = o(1)$, forget rates seem to have little impact on performance. While our theoretical analysis (Section~\ref{sec:asym-hybrid}) recommends a learning rate between $\zeta_t = o(1)$ and $\zeta_t = \omega(1/t)$ for the non-averaged case, it seems pre-asymptotic performance is better for $\zeta_t = \Theta(1/t)$, so we may observe slightly better results when $\zeta_t \sim (1+p)/t$ for a not very high value of $p$, such as $p=2$.

\tsubimages[]{Comparison of SGD, IP, the IP-SGD hybrid, and their averaged counterparts on two different problems (see definition in Appendix~\ref{sec:problems}), similarly to Figure~\ref{fig:n1}. In the legend, ``Hybrid'' and ``aHybrid'' refer to the IP-SGD hybrid and its averaged version, respectively. In both problems we show the average of 1000 independent runs.}{n2}{
  \begin{center}
    \subimage[Problem \#1]{.45}{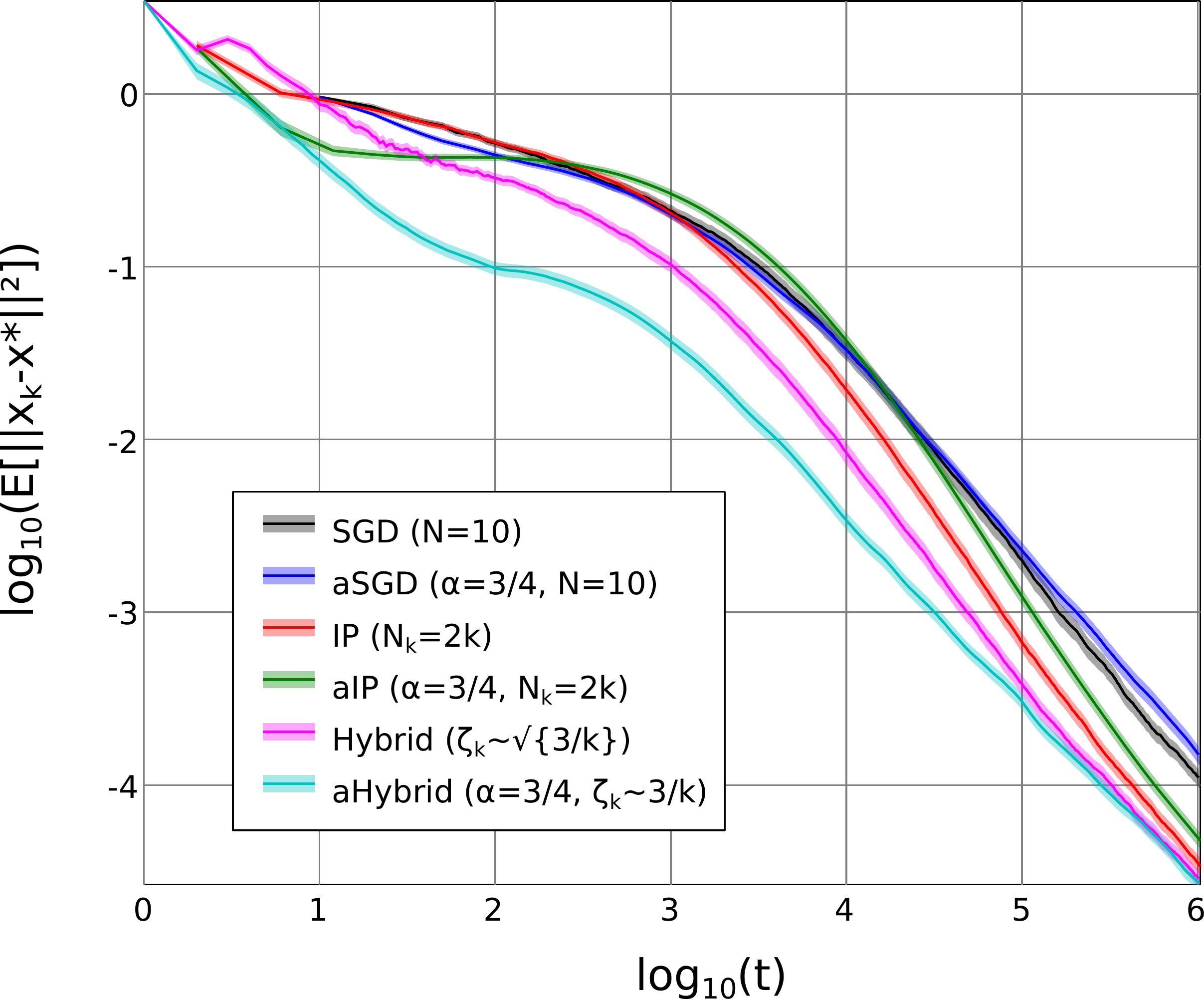}
    \subimage[Problem \#2, with $n=4$, $m=5$]{.45}{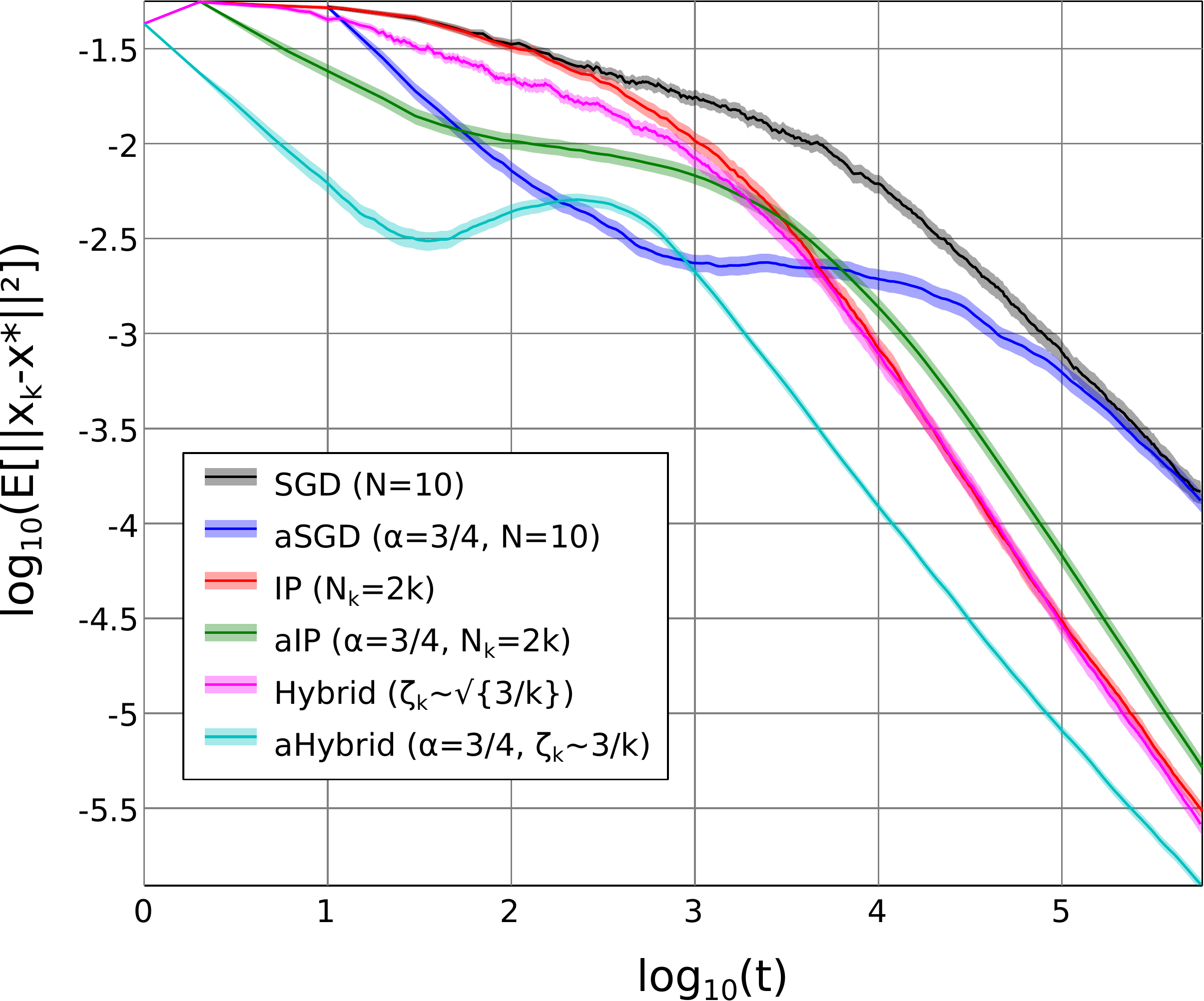}
  \end{center}
}

Figure~\ref{fig:n2} compares the IP-SGD hybrid to IP and SGD, as well as their averaged counterparts. All methods were preconditioned by a Gauss-Newton approximation of the Hessian at $x^*$. We observe that the IP-SGD hybrid provides better pre-asymptotic performance than SGD and IP, and similarly, the averaged hybrid provides better pre-asymptotic performance than averaged SGD and averaged IP. We also observe that SGD and aSGD appear to converge to the same asymptote, while the other four methods appear to converge to another lower asymptote. On (b), there is a more significant difference between the curves of these latter methods, most likely because the methods have not yet fully reached the asymptotic phase, and will eventually converge to the same asymptote later on.

\tsubimages[]{Comparison of AdaGrad, Adam and the IP-AdaGrad and IP-Adam hybrids on Problem \#1, showing the average of 1000 optimization runs, similarly to Figure~\ref{fig:n1}. The hybrid methods employed a forget rate of $\zeta_k=\frac{3}{k+2}$. In the captions, $\eta$ and $\alpha_k$ refer to learning rates following the notation of the original papers~\cite{adagrad,adam}. For Adam, we left the other hyperparameters in their default values ($\beta_1=.9, \beta_2=.999, \epsilon=10^{-8}$). AdaGrad was used in diagonal form.}{hybridcov}{
  \begin{center}
    \subimage[AdaGrad, $\eta=.05$]{.45}{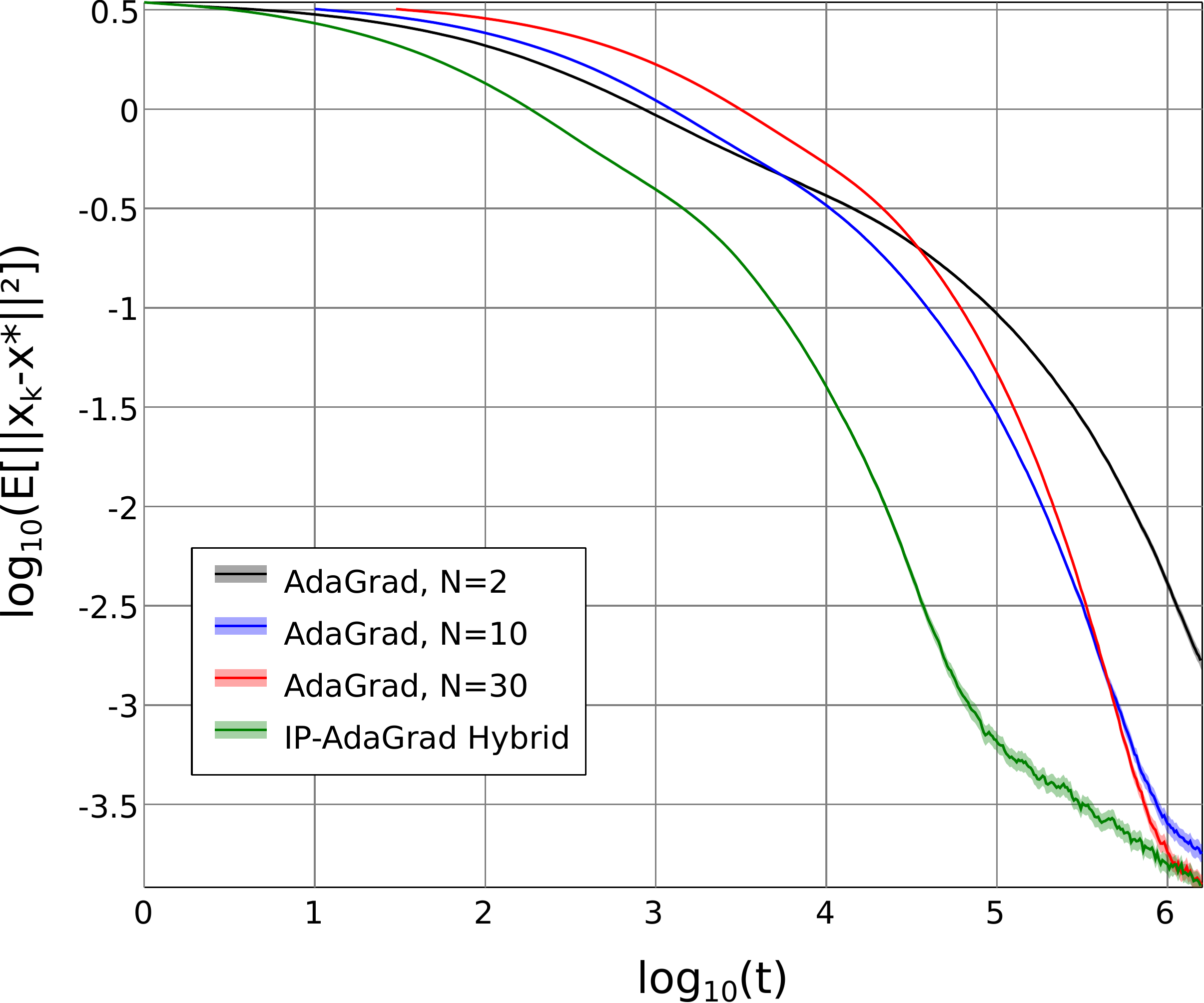}
    \subimage[AdaGrad, $\eta=.1$]{.45}{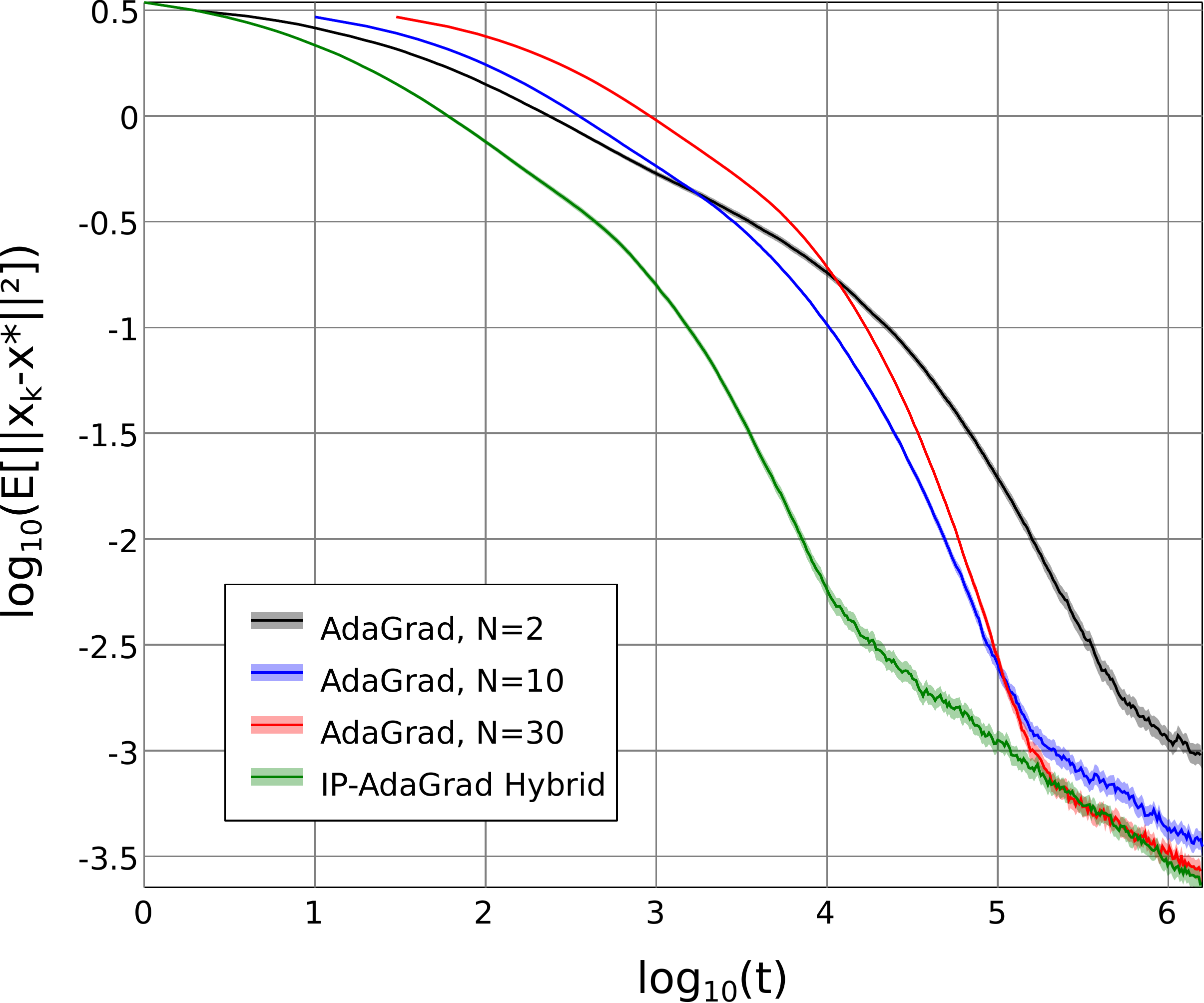}
  \end{center}
  \begin{center}
    \subimage[AdaGrad, $\eta=.5$]{.45}{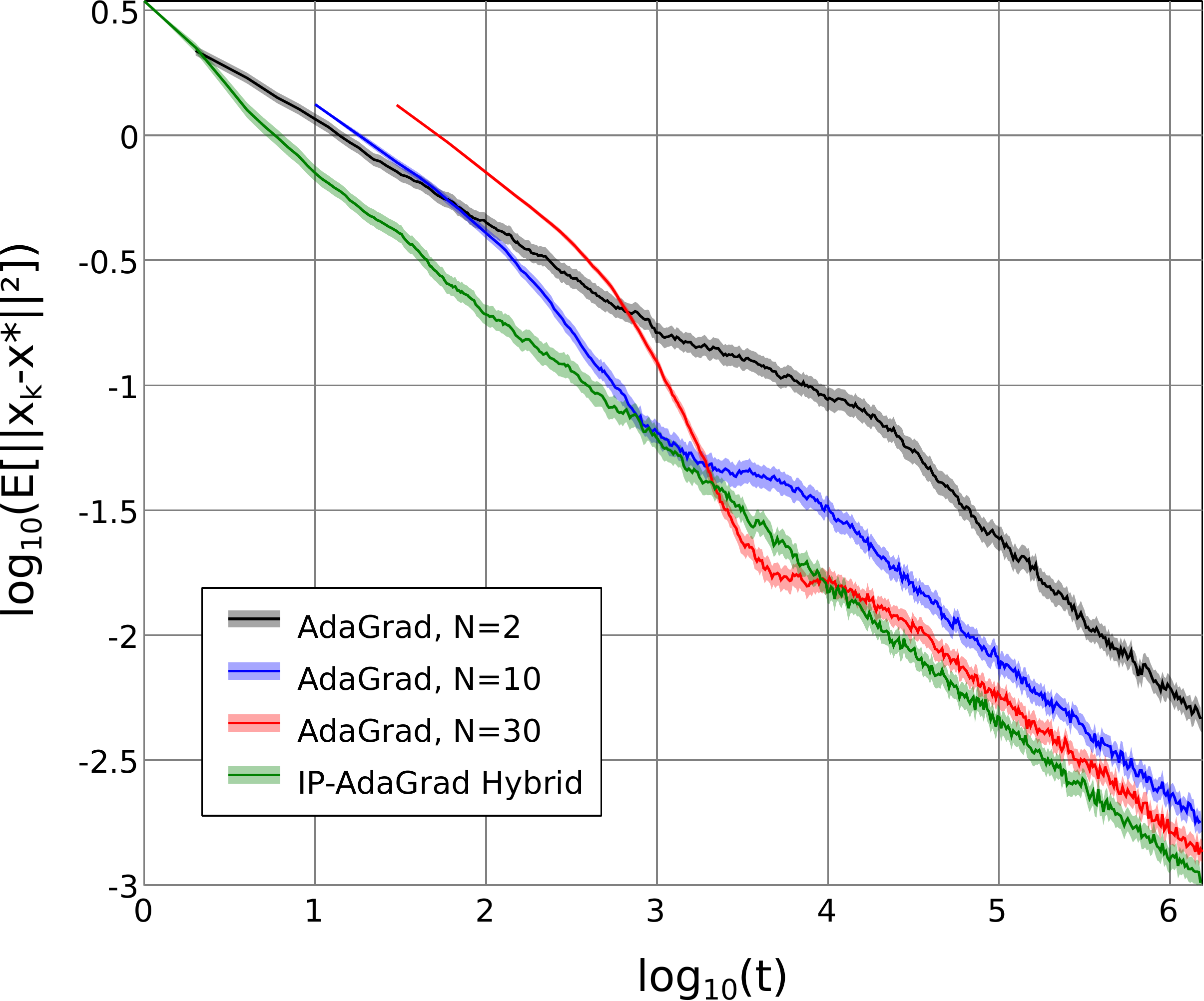}
    \subimage[Adam, $\alpha_k=.1/\sqrt{k}$]{.45}{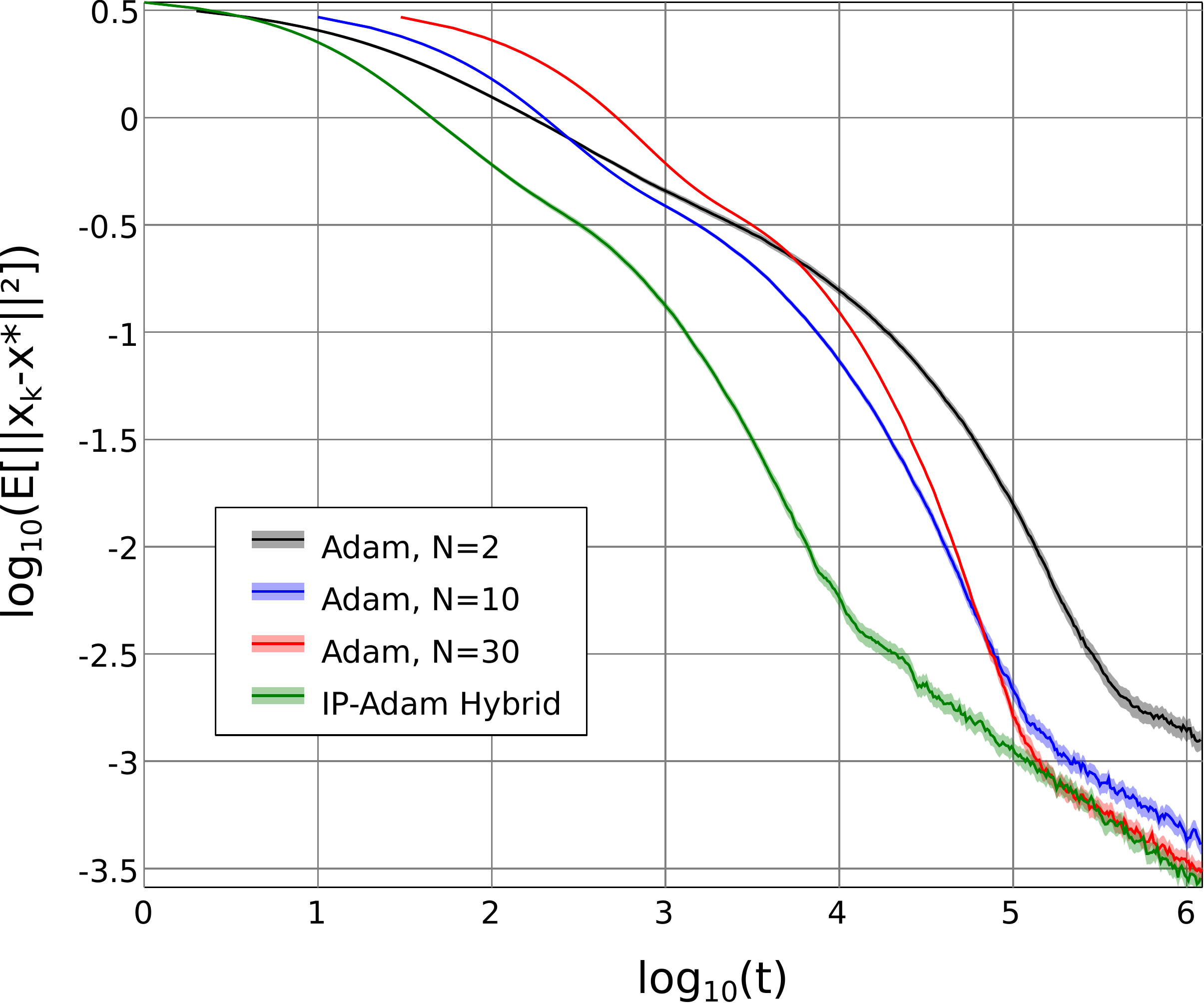}
  \end{center}
}

On Figure~\ref{fig:hybridcov}, we compare AdaGrad and Adam with their respective IP-hybrids, and observe that given a certain learning rate, the hybrid variant shows better performance than the original method for any $N$.\footnote{We remark that, in the studied cases, we observed that the choice of learning rate is quite independent of the choice of $N$, in the sense that it is reasonable to compare different $N$ under the same learning rate. This might not necessarily hold on other methods, or other decay patterns of the learning rate (e.g. as is the case of Adam with constant learning rate $\alpha_k = const.$).} However, it is worth reminding that methods like AdaGrad and Adam were designed to minimize a different type of metric (``\textit{regret}''), that is not very meaningful on MCLS problems. Thus, the asymptotic performance of these methods is much slower than Hessian-preconditioned SGD or IP, with $E[||x_k-x^{*}||^2] = O(1/\sqrt{k})$ when considering a constant learning rate $\eta$ for AdaGrad, or a learning rate decay of $\alpha_k = \Theta(1/\sqrt{k})$ for Adam.

\section{Incorporating a Gauss-Newton preconditioner}
\label{sec:gn}

While we obtain interesting results when we apply the hybrid approach to state-of-art stochastic optimization methods designed for LSL such as AdaGrad and Adam, what might be most effective for MCLS is to combine the IP-SGD hybrid with a Gauss-Newton based preconditioner, built from the previous Jacobian samples $\uhat J_t$ received by the algorithm. As the Gauss-Newton method of optimization provides an approximation of the Hessian matrix, coupling SGD, IP, or the IP-SGD hybrid with a Gauss-Newton matrix should provide a similar behavior to their Hessian-preconditioned counterparts, as long as the Gauss-Newton approximation is accurate. Employing dynamic Hessian estimation through Gauss-Newton or Quasi-Newton should thus make it easier to choose the learning rates (step size sequence) in relation to covariance-preconditioning as in AdaGrad and Adam, because in the former methods the optimal scale for the gradient steps is already embedded in the Hessian approximation, while the latter still require this fine-tuning.

Gauss-Newton might also be a more attractive option than Quasi-Newton based stochastic optimization methods, as Quasi-Newton faces many complications in the stochastic scenario. The classical Quasi-Newton methods of optimization use finite differencing to calculate the Hessian, which is extremely unstable in stochastic optimization\footnote{Older methods, from the stochastic approximation literature, mitigate this issue by reducing the finite difference bias sufficiently slowly~\cite{Wei,SA-GN,SPSA}, while the more recent ones, targeted at LSL applications, take one extra sample of the gradient while reusing pseudorandom numbers when finite differencing~\cite{schraudolph2007,bordes2009,wang2017}. In LSL, this means taking finite differences from two gradients computed from the same mini-batch.}, and use line search to guarantee its positive definiteness, which is difficult to implement in the stochastic scenario\footnote{Most methods fail to guarantee positive-definiteness of the Hessian estimate. Exceptions include Bordes et al.'s~\cite{bordes2009} and Wang et al.'s~\cite{wang2017} methods.}. Gauss-Newton, on the other hand, does not require either, which is the reason we believe it is capable of providing much better results than Quasi-Newton.

We employ the IP-SGD hybrid (Equation~\ref{eq:ip-sgd-hybrid}) with $A_t \sim B_t^{-1}/t$, where $B_t^{-1}$ dynamically estimates the inverse Gauss-Newton matrix. Ideally, we would like to use an unbiased estimator of $(J^TJ)^{-1}$, but such an estimator is most likely impossible to generate\footnote{Unbiased estimators for the inverse of the mean are only available for very specific cases, such as a Gaussian distribution of known variance~\cite{inv-mean}.}. Therefore, instead, we use the following approximation:
\lqq B_t = \frac{G_t^TG_t+R_t}{\left(\sum_{i=1}^{t-1} \left(1 - \frac{q_i}{q_{t-1}}\right)\right)^2 + t - 1}, \text{ }\text{ }\text{ } G_t=\sum_{i=1}^{t-1} \left(1 - \frac{q_i}{q_{t-1}}\right)\uhat J_i, \text{ }\text{ } R_t = \sum_{i=1}^{t-1} \uhat J_i^T\uhat J_i, \rqq
which can also be written using forget rates ($\zeta_i = \frac{q_i}{\sum_{j=1}^i q_j}$) as:
\eq{B_t = \frac{G_t^TG_t+R_t}{\left(t - 1 - 1/\zeta_{t-1}\right)^2 + t - 1}, \text{ }\text{ }\text{ } G_t=\left(\sum_{i=1}^{t-1} \uhat J_i \right) - \bar J_t/\zeta_{t-1}, \text{ }\text{ } R_t = \sum_{i=1}^{t-1} \uhat J_i^T\uhat J_i. \label{eq:sgn}}

We call this the \textit{stochastic Gauss-Newton} (SGN) method. Additionally, because in the first iterations $B_t$ may be ill-posed, we replace the update rule with $x_{t+1} = x_t$ until $B_t$ can be inverted. 

The purpose of the $\left(1 - \frac{q_i}{q_{t-1}}\right)$ factor in $G_t$ is to give the opposite weight from the gradient estimator of Equation~\ref{eq:hybrid1}. That is, while the gradient estimator gives higher weights to more recent measurements $(\uhat Q_i, \uhat J_i)$, gradually forgetting older samples, the Gauss-Newton estimator gives higher weight to those ``forgotten'' samples. Meanwhile, the regularization term $R_t$ is necessary for stability. If we used simply $B_t \propto G_t^TG_t$, the distribution of $B_t^{-1}$ could have very few finite moments, leading to instability (Note that our convergence analysis in Theorems \ref{thm:ip} and \ref{thm:ipa} requires that the gradient estimator has a minimum number of finite moments). While this can be solved simply by Tikhonov regularization, e.g. $B_t \propto G_t^TG_t + \gamma I$, the $R_t$ term defined above provides a similar regularization effect\footnote{Refer to Section~\ref{sec:stability} for details.} without requiring to arbitrate $\gamma$.

We may also incorporate Ruppert-Polyak averaging to SGN, by taking the update rule
\lqq \tilde x_{t+1} = \tilde x_t - \frac{B_t^{-1}}{t^\alpha} g_t, \text{ }\text{ }\text{ } x_{t+1} = \frac{\sum_{i=1}^{t} \tilde x_{i+1}}{t}, \rqq
where $g_t$ is computed according to Equation~\ref{eq:fuzzy} (and with $\uhat Q_t = \uhat Q(\tilde x_t)$, $\uhat J_t = \uhat J(\tilde x_t)$), $B_t$ computed according to Equation~\ref{eq:sgn}, and $\frac12 < \alpha < 1$. We call this variant \textit{averaged SGN} (aSGN).

\subsection{Experiments}
\label{sec:gn-exp}

We compared the performance of SGN and averaged SGN to two stochastic Quasi-Newton methods~\cite{bordes2009,wang2017}, the IP-Adagrad and IP-Adam hybrids, and a Ruppert-Polyak averaged implementation of the IP-Adam hybrid\footnote{When incorporating Ruppert-Polyak averaging to the IP-Adam hybrid, we mean taking the average of all iterates, when using a learning rate of $\alpha_k = \Theta(1/k^a)$, with $1/2 < a < 1$, analogously to the averaged IP-SGD hybrid or aSGN. We did not employ exponential averaging as the authors of Adam recommend~\cite{adam} as it does not have the same asymptotic properties as standard Ruppert-Polyak averaging.} on four different problems. The AdaGrad and Adam variants were all implemented in diagonal form. The Quasi-Newton methods required a few modifications in order to obtain a reasonable performance. While Bordes et al. originally constrain that the inverse Hessian entries must be above a threshold $.01/\lambda$, we further constrain that they must be also below $100/\lambda$. Additionally, we did not include their skipping strategy as it is not necessary here. As of Wang et al.'s method, we do not update the Hessian matrix when the finite difference vector is $y=0$, and for simplicity we used the BFGS version instead of the L-BFGS version of the method, as dimensionality is not a problem here. Also, although the Quasi-Newton methods require two gradient measurements per iteration, we counted only one when comparing performance, although this does not affect our conclusion.

\let\oldarraystretch=\arraystretch
\renewcommand{\arraystretch}{1.5}
\begin{table}
\centering
\caption{Choice of hyperparameters for the experiments of Figure~\ref{fig:sgn}.}
\label{table:hyperparameters}
\begin{tabular}{| c | c | c | c | c |}
\hline
 & (a)  & (b) & (c)  & (d) \\
\hline
SGN & \multicolumn{4}{c |}{$a_t = \frac1{t}$, $\zeta_t = \sqrt{\frac{3}{t+2}}$} \\[2pt]
\hline
aSGN & \multicolumn{4}{c |}{$a_t = \frac1{t^{.75}}$, $\zeta_t = \frac{3}{t+2}$} \\[2pt]
\hline
Bordes2009 & \multicolumn{2}{c |}{$\lambda=.1, t_0=1$} & \multicolumn{2}{c |}{$\lambda=.01, t_0=1$} \\[2pt]
\hline
Wang2017 & \multicolumn{4}{c |}{$\eta_k = \frac{1}{k}$, $H_1 = 10^{-10}I$} \\[2pt]
\hline
IP-AdaGrad Hybrid & \multicolumn{4}{c |}{$\eta = .1$, $\zeta_t = \frac{3}{t+2}$} \\[2pt]
\hline
IP-Adam Hybrid & \multicolumn{4}{c |}{$\alpha_t = \frac{.1}{\sqrt{t}}$, $\zeta_t = \frac{3}{t+2}$} \\[2pt]
\hline
Avg. IP-Adam Hybrid & \multicolumn{4}{c |}{$\alpha_t = \frac{1}{t^{.75}}$, $\zeta_t = \frac{3}{t+2}$} \\[2pt]
\hline
\end{tabular}
\end{table}
\let\arraystretch=\oldarraystretch

The hyperparameters for each method were chosen manually and are listed on Table~\ref{table:hyperparameters}. The learning rates of the AdaGrad and Adam variants ($\eta$ and $\alpha_t$, respectively) and the $\lambda$ parameter of Bordes et al.'s method were selected by a non-exhaustive search, while all other hyperparameters were set to their default values. As observed in Section~\ref{sec:hybrid-exp}, forget rates $\zeta_t$ do not significantly impact performance and were therefore not fine-tuned, while the learning rates of Gauss-Newton and Quasi-Newton variants are automatically determined by the theory.

\tsubimages[t]{Comparison between proposed methods (SGN, aSGN), existing Quasi-Newton approaches (Bordes et al.~\cite{bordes2009}, Wang et al.~\cite{wang2017}), and adaptive gradient methods (AdaGrad~\cite{adagrad}, Adam~\cite{adam}). In all cases we show the average error of 1000 independent runs. We used $N=20$ samples per iteration to compute gradients $\widehat{\nabla f}_N$ in all non-hybrid methods (Bordes2009 and Wang2015). As in Figure~\ref{fig:n1}, the error margin of each configuration is denoted with a translucent fill.}{sgn}{
  \begin{center}
    \subimage[Problem \#1]{.45}{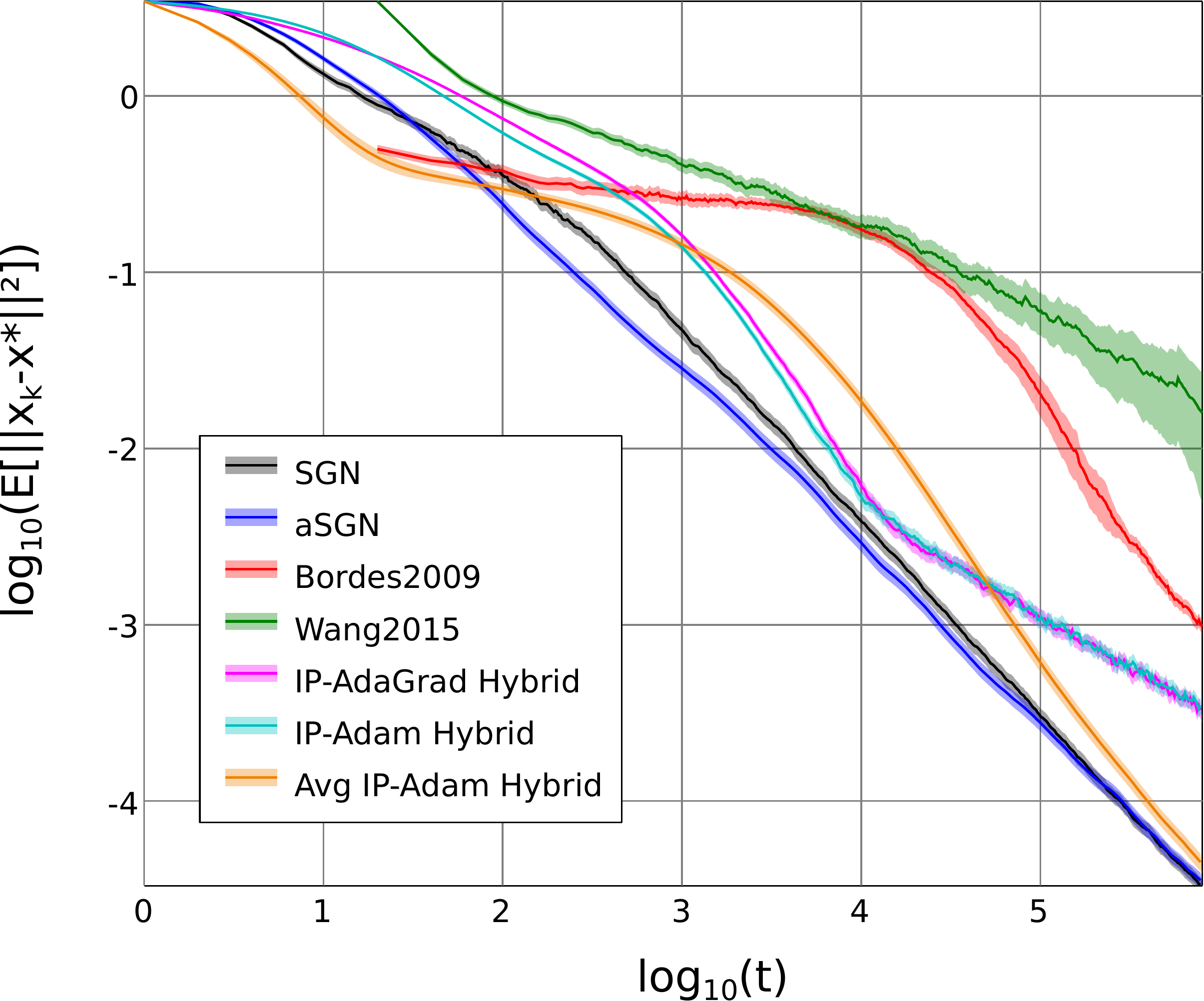}
    \subimage[Problem \#2, with $n=10$, $m=25$]{.45}{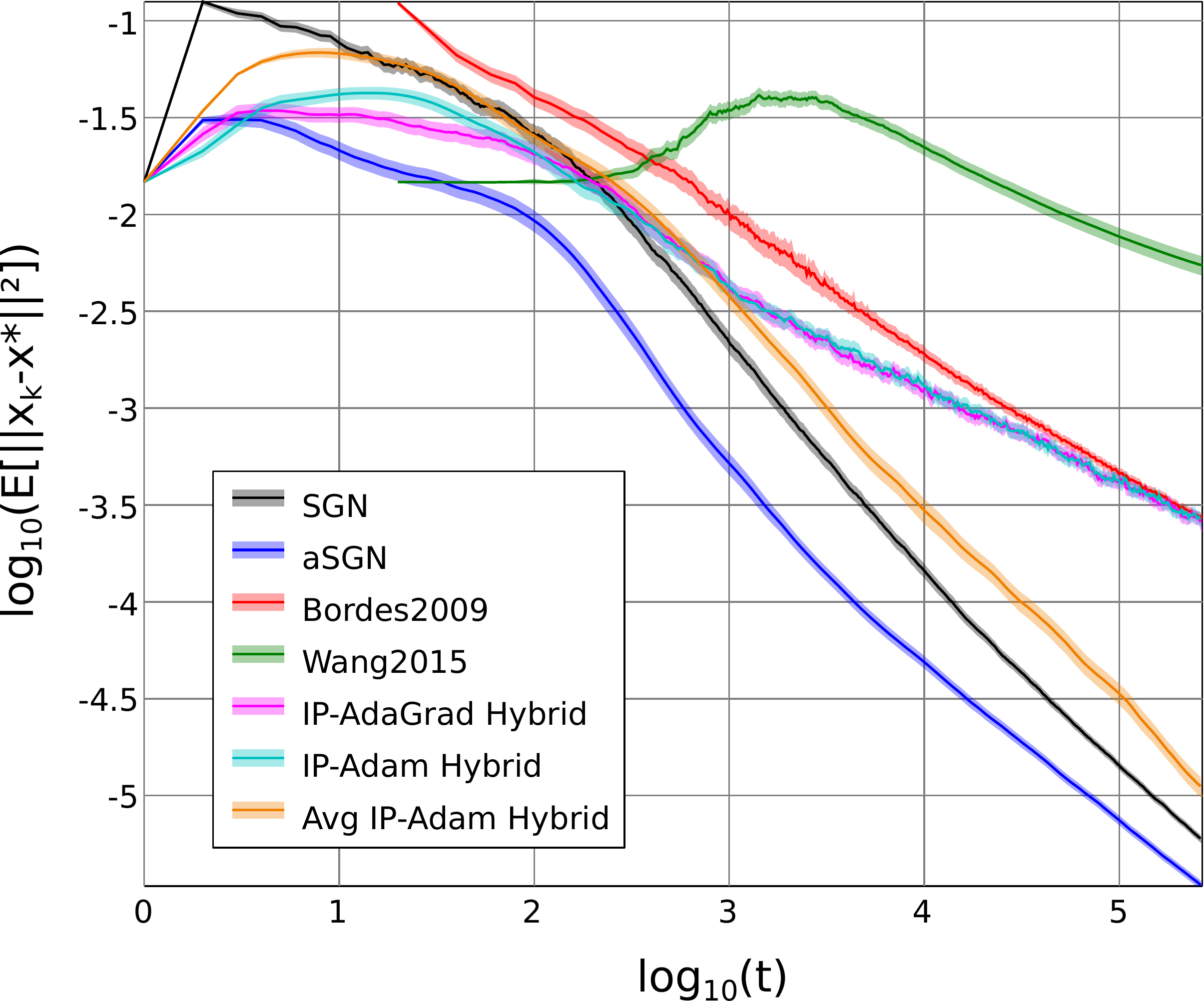}
  \end{center}
  \begin{center}
    \subimage[Problem \#2, with $n=10$, $m=20$]{.45}{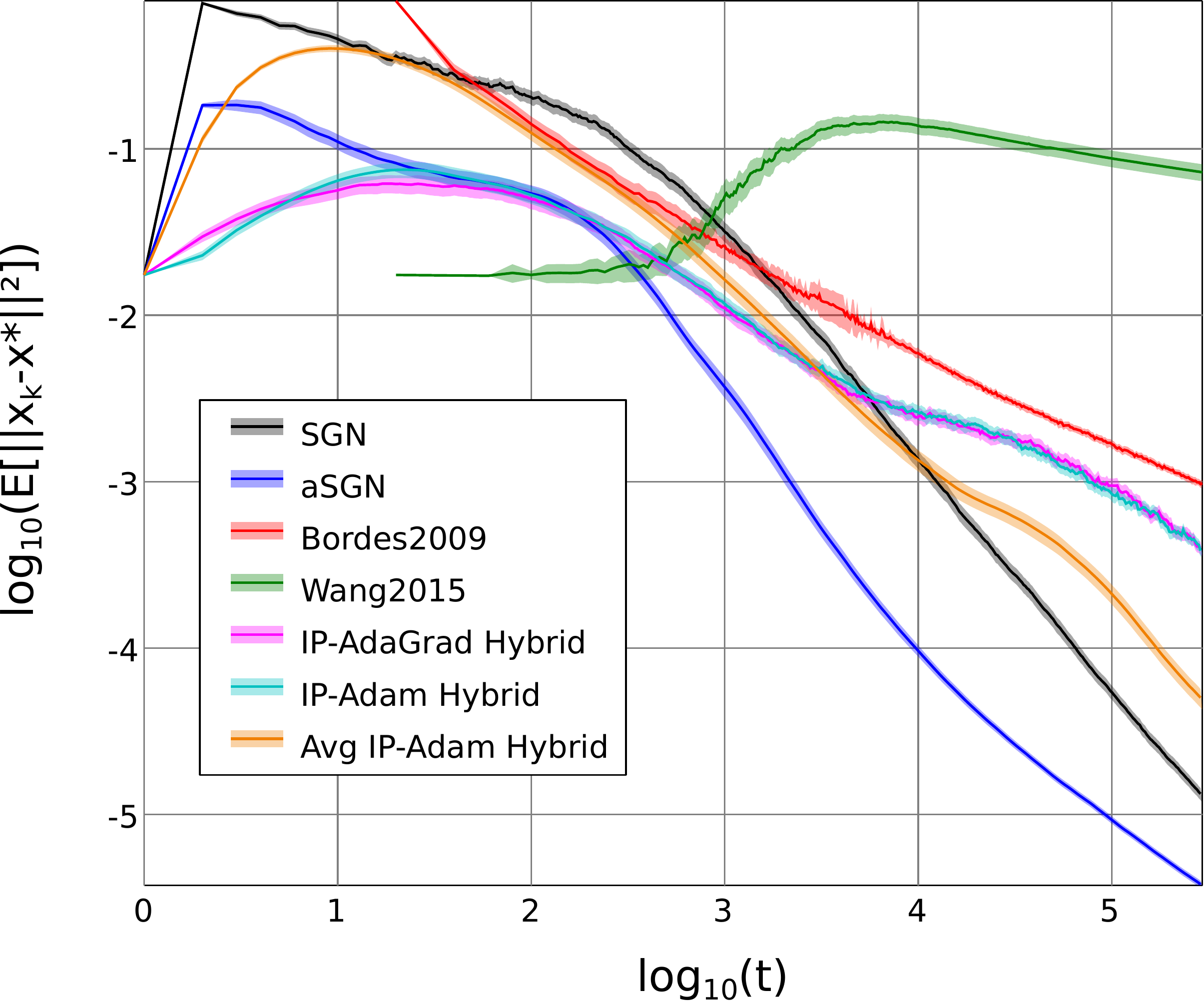}
    \subimage[Problem \#2, with $n=10$, $m=15$]{.45}{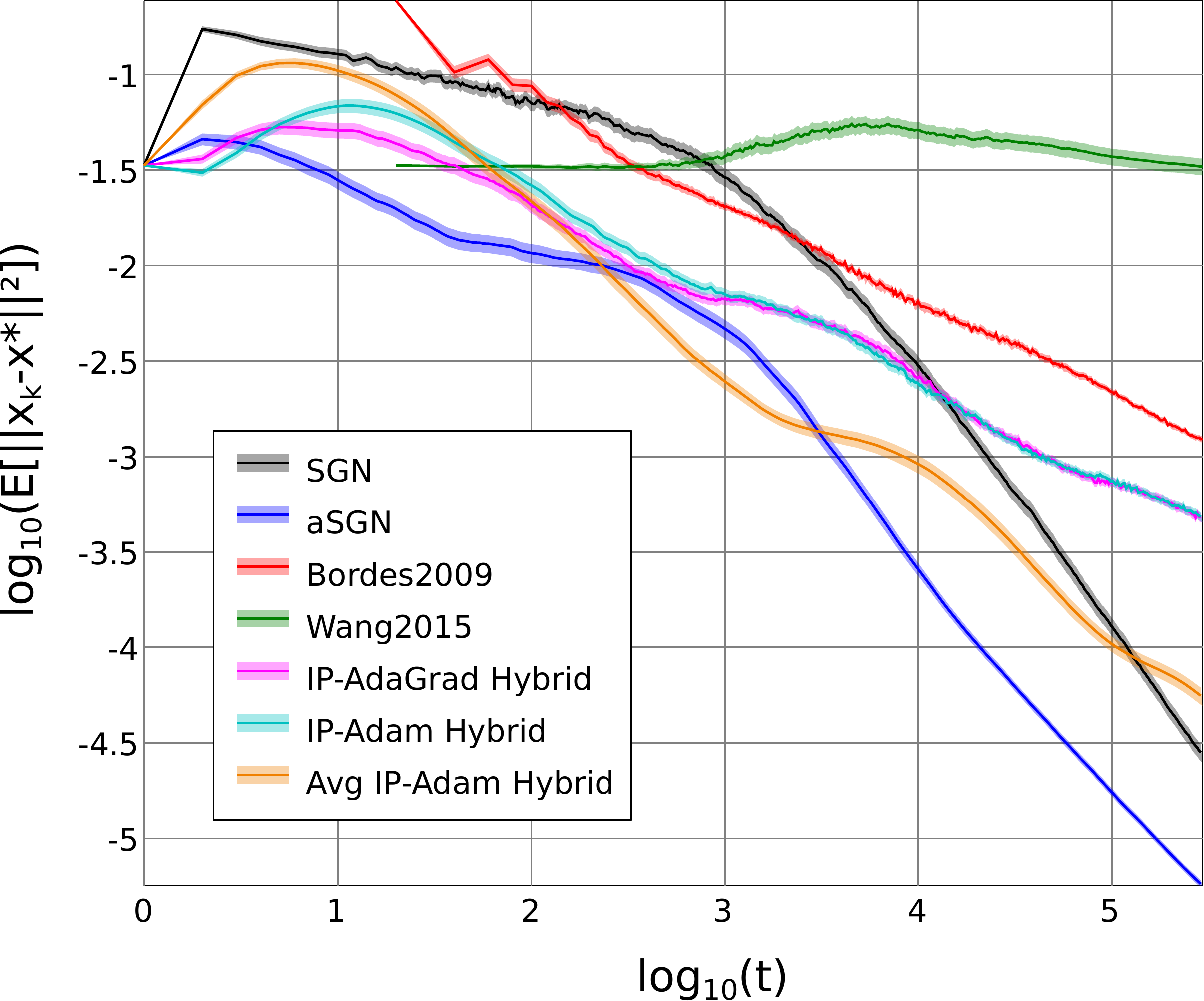}
  \end{center}
}

Figure~\ref{fig:sgn} shows how the Quasi-Newton approaches perform very poorly on MCLS, probably due to their inconsistent Hessian estimators. The IP-AdaGrad and IP-Adam hybrids have almost identical behavior, converging to the same $\Theta(1/\sqrt{t})$ asymptote. SGN, aSGN and the averaged IP-Adam hybrid exhibit the best results, and it is clear in (a) that they converge to the same asymptote. Although Ruppert-Polyak averaging guarantees optimal asymptotic behavior regardless of the choice of preconditioner, we observe that in all four cases (a-d) pre-asymptotic behavior is generally better for Gauss-Newton preconditioning (aSGN) than for covariance preconditioning (averaged IP-Adam hybrid).

\section{Limitations and generalizations}
\label{sec:limitations}
Our cost assumptions (Section~\ref{sec:assumptions}) might pose some constraints on what kinds of problems the presented methods may be applied to:
\begin{itemize}
\item When the problem dimensions $n,m$ are very high, we may not neglect the cost of tasks such as multiplying $J$ by $Q$ or inverting $B_t$. Using sparse matrices may mitigate for the former issue, while the latter may require updating $B_t$ less often. Although the stochastic Gauss-Newton approach is probably not very effective for very high $n$, due to the $O(n^3)$ cost of inverting $B_t$, it is not rare for MCLS problems to have very small $n$ (e.g. 9 parameters in~\cite{pramook}).
\item When $(\uhat Q, \uhat J)$ are uncorrelated, there is no need to impose $i\neq j$ on Equation~\ref{eq:gradprec}. Nevertheless, IP remains asymptotically faster than SGD. The same applies to the hybrid methods.
\item Similarly, when the cost of computing a $(\uhat Q, \uhat J)$ pair is considerably higher than computing $\uhat Q$ alone, Equation~\ref{eq:gradprec} might not be the most efficient way to estimate the gradient, in the sense that we may want to compute more samples of $\uhat Q$ than $\uhat J$ per iteration. However, gradually increasing the number of samples per iteration is still asymptotically more efficient than keeping it constant. The hybrid method, however, would have to be modified to accommodate this more sophisticated gradient estimator.
\item We believe IP can be modified to handle a gradient estimator that is biased but consistent (i.e. bias goes to zero as $N\rightarrow \infty$), although this would require a more careful convergence analysis, which we leave for future work. Note that none of the literature methods considered in this paper, as well as the hybrid methods, would be able to converge to the correct minimum when the gradient estimator is biased.
\end{itemize}

\section{Theoretical analysis}
\label{sec:analysis}

This section groups the more detailed theoretical results and their proofs regarding the proposed methods. Section~\ref{sec:theorems} proves convergence properties of IP and aIP on strongly convex problems, while Section~\ref{sec:asym-hybrid} analyzes the hybrid methods in a more limited, linear scenario. Section~\ref{sec:stability} provides some theoretical support to justify the regularization scheme of our Gauss-Newton methods.

\subsection{Convergence analysis of IP and averaged IP}
\label{sec:theorems}

In this section, we prove the convergence properties of IP and aIP for strongly convex problems.

Let $\mathcal{F}_k$ be an increasing sequence of $\sigma$-algebras, where $x_1\in\mathbb{R}^n$ is an $\mathcal{F}_0$-measurable random variable, and consider the the update rule:
\begin{equation}
x_{k+1} = x_k - A_k g_k \label{eq:tur}
\end{equation}
where $A_k \in \mathbb{R}^{n\times n}$ a deterministic sequence of positive matrices, and $g_k = \nabla f(x_k) + \mathcal{E}_k$ is an unbiased estimator of $\nabla f(x_k)$, for some function $f:\mathbb{R}^n\rightarrow\mathbb{R}$. That is, we assume $\mathcal{E}_k\in\mathbb{R}^n$ is an $\mathcal{F}_k$-measurable random variable satisfying $E[\mathcal{E}_k|\mathcal{F}_{k-1}] = 0$ whose conditional distribution with respect to $\mathcal{F}_{k-1}$ is a function of $x_k$ and $k$, i.e. $E[g(\mathcal{E}_k)|\mathcal{F}_{k-1}] = E[g(\mathcal{E}_k)|x_k]$ for every function $g$ of $\mathbb{R}^n$. Further consider the following set of assumptions:
\begin{assumption}
\label{ass:convex}
$x^*$ is the only critical point of $f$; $\nabla^2 f(x)$ exists and is positive definite everywhere on $\mathbb{R}^n$, satisfying $\sup_x ||\nabla^2f(x)||_2 < +\infty$ and $\sup_x ||(\nabla^2f(x))^{-1}||_2 < +\infty$, where $||\cdot||_2$ is the induced L2 norm for matrices.
\end{assumption}
\begin{assumption}
\label{ass:moms}
$\sup_{x_k,k} N_k^{p/2}E[||\mathcal{E}_k||^p|\mathcal{F}_{k-1}] < +\infty,$ for all $p \in \{2, ..., M\}$, where $M \geq 2$ (to be specified later), for some positive nondecreasing sequence $N_k$.
\end{assumption}
\begin{assumption}
\label{ass:aknk}
$A_k = a_k D$, where $a_k\in\mathbb{R}$ is of the form $a_k = (k+c)^{-\alpha}$, for some $c\geq 0$ and $0 < \alpha \leq 1$, and $D$ is a positive definite matrix; while $N_k \sim N_1 k^q$, for some $N_1>0$, with $q\geq 0$.
\end{assumption}
\begin{assumption}
\label{ass:var}
The function $\tilde U(x) = \lim_{k\rightarrow\infty} N_k E[\mathcal{E}_k\mathcal{E}_k^T|x_{k} = x]$ exists and is continuous in a neighborhood of $x^*$ with $\tilde U(x^*) = \Sigma^2$, for some positive matrix $\Sigma^2$.
\end{assumption}
\begin{assumption}
\label{ass:C3}
$f$ is $C^3$ in a neighborhood of $x^*$.
\end{assumption}

The convergence analysis is organized as follows. Lemma~\ref{lemma:convp} shows that the moments of $||x_k-x^*||$ converge fast enough, which will be used to prove almost sure (a.s.) convergence on Lemma~\ref{thm:as}. Once a.s. convergence is established, Theorem~\ref{thm:ip} calculates the convergence rate of IP by direct application of a theorem by Fabian~\cite{fabian1968}.

The convergence rate of aIP follows the same methodology of Ruppert~\cite{avg2}. By a.s. convergence, we compute a bound to the asymptotic decay of $||x_k-x^*||$ on Lemma~\ref{lemma:limsuplog}, which is used on Theorem~\ref{thm:ipa} to show that the difference between a linearized version of the problem (where $\nabla f(x)$ is a linear function) and the original nonlinear problem is negligible, so that the convergence rate of aIP is calculated based on the linearized problem.

\begin{lemma}
\label{lemma:convp}
Regarding Equation~\ref{eq:tur}, assume Assumptions \ref{ass:convex}, \ref{ass:moms} and \ref{ass:aknk}. Let $\gamma = \inf_{x\neq x^*} \langle x-x^*, \nabla f(x)\rangle/||x-x^*||_{D^{-1}}^2$, where $||u||_{D^{-1}}$ denotes $\sqrt{u^TD^{-1}u}$. Then $E[||x_k-x^*||_{D^{-1}}^{2p}] = O(B_k^p)$ as $k\rightarrow\infty$, for all non-negative integer $p\leq M/2$, where 
\lqq B_k = \left\{\begin{matrix}
1/k^{\alpha+q}, & \text{if }0 < \alpha < 1 \\ 
1/k^{1+q}, & \text{if }\alpha = 1 \text{ and } 2\gamma > 1+q\\
(\log k)/k^{1+q}, & \text{if }\alpha = 1 \text{ and } 2\gamma = 1+q\\
1/k^{2\gamma}, & \text{if }\alpha = 1 \text{ and } 2\gamma < 1+q
\end{matrix}\right.. \rqq
\end{lemma}
\begin{proof}
We can write:
\lqq E[||x_{k+1}-x^*||_{D^{-1}}^{2p}] = E\left[\left(||x_k-x^*-a_kD\nabla f(x_k) -a_kD\mathcal{E}_k||_{D^{-1}}^2\right)^p\right] \rqq
\lqq  = E\left[\left(||x_k-x^*-a_kD\nabla f(x_k)||_{D^{-1}}^2 - ... \right.\right. \rqq
\lqq  ...\left.\left. 2a_k\langle x_k-x^*-a_kD\nabla f(x_k), D\mathcal{E}_k\rangle_{D^{-1}} + a_k^2||D\mathcal{E}_k||_{D^{-1}}^2\right)^p\right] \rqq
\lqq  \leq E\left[||x_k-x^*-a_kD\nabla f(x_k)||_{D^{-1}}^{2p}\right] - ... \rqq
\lqq ... 2pa_kE\left[\langle x_k-x^*-a_kD\nabla f(x_k), D\mathcal{E}_k\rangle_{D^{-1}}||x_k-x^*-a_kD\nabla f(x_k)||_{D^{-1}}^{2p-2}\right] + ... \rqq
\lqq ...\sum_{2\leq i\leq 2p} \binom{2p}{i} E\left[||x_k-x^*-a_kD\nabla f(x_k)||_{D^{-1}}^{2p-i}.a_k^{i}||D\mathcal{E}_k||_{D^{-1}}^i\right]. \rqq

The second term ($2pa_kE[...]$) in the expression above is zero since $E[\mathcal{E}_k|x_k] = 0$, giving us
\lqq E[||x_{k+1}-x^*||_{D^{-1}}^{2p}] \leq E\left[||x_k-x^*-a_kD\nabla f(x_k)||_{D^{-1}}^{2p}\right] + ... \rqq
\lqq ...\sum_{2\leq i\leq 2p} \binom{2p}{i} E\left[||x_k-x^*-a_kD\nabla f(x_k)||_{D^{-1}}^{2p-i}.a_k^{i}||D\mathcal{E}_k||_{D^{-1}}^i\right]. \rqq

Suppose by induction that the lemma is true for $p-1$. The base case ($p=0$) is trivial.

Note then that for $2 \leq i \leq 2(p-1)$, we may simplify:
\begin{align}
E\left[||x_k-x^*-a_kD\nabla f(x_k)||_{D^{-1}}^{2p-i}.a_k^{i}||D\mathcal{E}_k||_{D^{-1}}^i\right]&= \nonumber \\
O\left(E\left[||x_k-x^*-a_kD\nabla f(x_k)||_{D^{-1}}^{2p-i}.a_k^{i}/N_k^{i/2}\right]\right)&= \nonumber \\
O\left(E\left[(1+O(a_k))||x_k-x^*||_{D^{-1}}^{2p-i}.a_k^i/N_k^{i/2}\right]\right)&= \nonumber \\
O\left(E\left[||x_k-x^*||_{D^{-1}}^{2p-i}.a_k^i/N_k^{i/2}\right]\right)&, \nonumber
\end{align}
which, because $E[X^h] \leq E[X^r]^{h/r}$ for any random variable $X \in \mathbb{R}_+$ where $0 < h \leq r$, can be bounded to:
\lqq O\left(E\left[||x_k-x^*||_{D^{-1}}^{2p-1}\right]^{\frac{2p-i}{2p-1}}.a_k^i/N_k^{i/2}\right)= O\left(B_k^{p-i/2}.a_k^i/N_k^{i/2}\right), \rqq
and thus:
\lqq E[||x_{k+1}-x^*||_{D^{-1}}^{2p}] \leq E\left[||x_k-x^*-a_kD\nabla f(x_k)||_{D^{-1}}^{2p}\right] +... \rqq
\lqq ... \sum_{2\leq i\leq 2p} \binom{2p}{i} O\left(B_k^{p-i/2}a_k^{i}/N_k^{i/2}\right).\rqq

Then, because $B_k^{-1} = O(N_k/a_k)$ implies $\sum_{2\leq i\leq 2p} O\left(B_k^{p-i/2}a_k^{i}/N_k^{i/2}\right) = \sum_{2\leq i\leq 2p} O\left(B_k^{p-1}a_k^{i/2+1}/N_k\right) = O\left(B_k^{p-1}a_k^{2}/N_k\right)$, we may write:
\lqq E[||x_{k+1}-x^*||_{D^{-1}}^{2p}] \leq E\left[||x_k-x^*-a_kD\nabla f(x_k)||_{D^{-1}}^{2p}\right] + O\left(a_k^2B_k^{p-1}/N_k\right) \rqq
\lqq \leq (1-2p\gamma a_k+O(a_k^2))E\left[||x_k-x^*||_{D^{-1}}^{2p}\right] + O\left(a_k^2B_k^{p-1}/N_k\right). \rqq

The solution to this recurrence is:
\lqq E[||x_{k+1}-x^*||_{D^{-1}}^2] = O\left(\sum_{j=1}^ke^{\sum_{i=j+1}^{k}- 2pa_i\gamma} a_j^2B_k^{p-1}/N_j\right) \rqq
\lqq = \left\{\begin{matrix}
O\left(k^{\alpha} \cdot k^{-\alpha(p+1)-qp}\right), & \text{if }0 < \alpha < 1 \\ 
O\left(\sum_{j=1}^k\frac{j^{2\gamma p}}{k^{2\gamma p}} j^{-(p+1)-qp}\right), & \text{if }\alpha = 1 \text{ and } 2\gamma > 1+q\\
O\left(\sum_{j=1}^k\frac{j^{2\gamma p}}{k^{2\gamma p}} (\log p)^{p-1}j^{-(p+1)-qp}\right), & \text{if }\alpha = 1 \text{ and } 2\gamma = 1+q\\
O\left(\sum_{j=1}^k\frac{j^{2\gamma p}}{k^{2\gamma p}} j^{-2-q-2\gamma(p-1)}\right), & \text{if }\alpha = 1 \text{ and } 2\gamma < 1+q
\end{matrix}\right. \rqq
\lqq =O(B_k^p), \rqq
which proves the lemma.
\end{proof}

\begin{lemma}
\label{thm:as}
Regarding Equation~\ref{eq:tur}, assume Assumptions \ref{ass:convex}, \ref{ass:moms} and \ref{ass:aknk} with $M \geq 2\lfloor \frac1{\alpha+q} \rfloor+2$ if $\alpha < 1$, or $M \geq 2\lfloor \frac1{\min\{1+q,2\gamma\}} \rfloor+2$ if $\alpha=1$, where $\gamma = \inf_{x\neq x^*} \langle x-x^*, \nabla f(x)\rangle/||x-x^*||_{D^{-1}}^2$. Then $x_k \rightarrow x^*$ almost surely.
\end{lemma}
\begin{proof}
Almost sure convergence is equivalent to:
\begin{equation}
(\forall \delta>0)\text{ } \lim_{k_0 \rightarrow \infty} P\left[\sup_{k\geq k_0} ||x_k - x^*|| \geq \delta \right] = 0. \label{eq:as}
\end{equation}

Now,
\lqq P\left[\sup_{k\geq k_0} ||x_k - x^*|| \geq \delta \right] = P\left[\bigvee_{k\geq k_0} ||x_k - x^*|| \geq \delta \right] \leq \rqq
\lqq \sum_{k\geq k_0}P\left[||x_k - x^*|| \geq \delta \right] \leq \sum_{k\geq k_0}\frac{E\left[||x_k - x^*||^{2p} \right]}{\delta^{2p}} \leq \sum_{k\geq k_0}\frac{||D||_2^p E\left[||x_k - x^*||_{D^{-1}}^{2p} \right]}{\delta^{2p}}, \rqq
which by Lemma~\ref{lemma:convp}, satisfies for $p\leq\lfloor M/2\rfloor$:
\lqq P\left[\sup_{k\geq k_0} ||x_k - x^*|| \geq \delta \right] \leq \sum_{k\geq k_0}\frac{O(B_k^p)}{\delta^{2p}}. \rqq

However, if $M$ is high enough ($M \geq 2\lfloor \frac1{\alpha+q} \rfloor+2$ if $\alpha < 1$, or $M \geq 2\lfloor \frac1{\min\{1+q,2\gamma\}} \rfloor+2$ if $\alpha=1$), there exists $p$ such that the summation above converges, which means
\lqq \lim_{k_0 \rightarrow \infty} \sum_{k\geq k_0}\frac{O(B_k^p)}{\delta^{2p}} = 0, \rqq
thus proving Equation~\ref{eq:as}.
\end{proof}

\begin{theorem}
\label{thm:ip}
Regarding Equation~\ref{eq:tur}, assume Assumptions \ref{ass:convex}, \ref{ass:moms}, \ref{ass:aknk} and \ref{ass:var}, with $\alpha=1$, $D=(1+q)S^{-1}$, where $S=\nabla f^2(x^*)$, and $M \geq 2\lfloor \frac1{\min\{1+q,2\gamma\}} \rfloor+2$, where $\gamma = \inf_{x\neq x^*} \langle x-x^*, \nabla f(x)\rangle/||x-x^*||_{D^{-1}}^2$. Then $E[(x_{k+1}-x^*)(x_{k+1}-x^*)^T] \sim \frac{S^{-1}\Sigma^2S^{-1}}{\sum_{i=1}^k N_k}$.
\end{theorem}
\begin{proof}
By Lemma~\ref{thm:as}, $x_k\rightarrow x^*$ almost surely. Then we may apply a theorem by Fabian (Theorem 2.2 in \cite{fabian1968}) with (Fabian's notation in the left side, our notation in the right) $U_k = x_k-x^*$, $\Gamma_k=D\int_{\tau=0}^1\nabla^2f(x^*+\tau(x-x^*))d\tau$, $\alpha=\alpha$, $\Phi_k=D$, $\beta=q+\alpha$, $\Sigma=\Sigma^2/N_1$, $V_k=k^{-q/2}\mathcal{E}_k$, $\Lambda = (1+q)I$, $\beta_+ = 1+q$ and $P=I$, which produces:
\lqq E[k^{(1+q)/2}(x_k-x^*)] \rightarrow 0, \text{ and} \rqq
\lqq \text{Var}[k^{(1+q)/2}(x_k-x^*)] \rightarrow (1+q)S^{-1}\Sigma^2S^{-1}/N_1, \rqq
implying
\lqq E[(x_{k+1}-x^*)(x_{k+1}-x^*)^T] \sim \frac{S^{-1}\Sigma^2S^{-1}}{\sum_{i=1}^k N_k}. \rqq
\end{proof}

\begin{lemma}
\label{lemma:limsuplog}
Regarding Equation~\ref{eq:tur}, assume Assumptions \ref{ass:convex}, \ref{ass:moms}, \ref{ass:aknk}, with $\alpha < 1$. Then,
\lqq \limsup_{k\rightarrow\infty} \frac{k^{\alpha+q}||x_k-x^*||^2}{k^h} = 0 \text{ (a.s.)} \rqq
for all $h > \frac1{\lfloor M/2 \rfloor}$.
\end{lemma}
\begin{proof}
The proof is established similarly as Lemma~\ref{thm:as}. By Lemma~\ref{lemma:convp}, for all $p \leq \lfloor M/2 \rfloor$, we may write:
\lqq P\left[\sup_{k\geq k_0} \frac{k^{\alpha+q}||x_k-x^*||^2}{k^h} \geq \delta^2 \right] \leq \sum_{k\geq k_0} \frac{k^{(\alpha+q)p}O(B_k^p)}{k^{hp}\delta^{2p}} = \sum_{k\geq k_0} \frac{O(1)}{k^{hp}\delta^{2p}},\rqq
which converges for $hp>1$ (hence for $h>\frac1{\lfloor M/2 \rfloor}$, there exists $p$ such that the summation above converges). Thus, $\frac{k^{\alpha+q}||x_k-x^*||^2}{k^h}$ converges to zero a.s., which proves the lemma.
\end{proof}

\begin{theorem}
\label{thm:ipa}
Regarding Equation~\ref{eq:tur}, assume Assumptions \ref{ass:convex}, \ref{ass:moms}, \ref{ass:aknk}, \ref{ass:var} and \ref{ass:C3}, $\frac{1-q}2 < \alpha < 1$, and $M$ that satisfies $M \geq 2\lfloor \frac1{\alpha+q} \rfloor+2$ and $M \geq 2\lfloor \frac1{\alpha-(1-q)/2} \rfloor+2$. Let
\lqq \bar x_{k+1} = \frac{\sum_{i=1}^k N_i x_{i+1}}{\sum_{i=1}^k N_i}. \rqq

Then $E[(\bar x_{k+1}-x^*)(\bar x_{k+1}-x^*)^T] \sim \frac{S^{-1}\Sigma^2S^{-1}}{\sum_{i=1}^k N_i}$ as $k\rightarrow\infty$.
\end{theorem}
\begin{proof}
The proof is similar to the one in Ruppert's work on the averaged Robbins-Monro procedure~\cite{avg2}.

Consider a linearized version of the problem, i.e. $x^{lin}$ and $\bar{x}^{lin}$ following
\lqq x^{lin}_{k+1} = x^{lin}_k - a_kD(S(x^{lin}_k-x^*) + \mathcal{E}_k), \text{ }\text{ }\text{ } \bar{x}^{lin}_{k+1} = \frac{\sum_{i=1}^k N_i x_{i+1}^{lin}}{\sum_{i=1}^k N_i}. \rqq
where $\mathcal{E}_k = g_k - \nabla f(x_k)$ (from the nonlinear process).

Subtracting the recurrences of the linear and nonlinear cases, we have
\begin{align}
x^{lin}_{k+1}-x_{k+1} &= (I-a_kDS)(x^{lin}_k-x_k) + a_kD(\nabla f(x_k-x^*)-S(x_k-x^*)) \nonumber \\
&= (I-a_kDS)(x^{lin}_k-x_k) + O(a_k||x_k-x^*||^2), \nonumber
\end{align}
which, by Lemma \ref{lemma:limsuplog}, implies
\lqq ||x^{lin}_{k+1}-x_{k+1}|| \leq ||I-a_kDS||_2||x^{lin}_k-x_k|| + Ck^{-2\alpha-q}k^h \rqq
for some $C$ after $k$ is sufficiently high, with $h > \frac1{\lfloor M/2 \rfloor}$. Then by Chung's lemma (\cite{chung}, Lemma 4),
\lqq ||x_k-x^{lin}_k|| = O((k^{-2\alpha-q}k^h)/a_k) = o(k^h/k^{\alpha+q}), \rqq
implying
\lqq ||\bar{x}_k-\bar{x}^{lin}_k|| = O\left(\frac{\sum_ii^qi^h/i^{\alpha+q}}{\sum_ii^q}\right) = o(k^h / k^{\alpha+q}) = o(k^{-(1+q)/2}), \rqq
since the constraint $M \geq 2\lfloor \frac1{\alpha-(1-q)/2} \rfloor+2$ allows for the existence of $h$ such that $h \leq \alpha - (1-q)/2$.

Thus, we can write:
\lqq E\left[(\bar{x}_{k+1}-x^*)(\bar{x}_{k+1}-x^*)^T\right] =  \rqq
\lqq E\left[\left(\bar{x}^{lin}_{k+1}-x^* + o\left(k^{-\frac{1+q}2}\right)\right)\left(\bar{x}^{lin}_{k+1}-x^* + o\left(k^{-\frac{1+q}2}\right)\right)^T\right] =  \rqq
\lqq E\left[(\bar{x}^{lin}_{k+1}-x^*)(\bar{x}^{lin}_{k+1}-x^*)^T\right] + o\left(k^{-\frac{1+q}2}E\left[\left\|\bar{x}^{lin}_{k+1}-x^*\right\|\right]\right) + o(k^{-(1+q)}) \rqq
\begin{equation}
= E\left[(\bar{x}^{lin}_{k+1}-x^*)(\bar{x}^{lin}_{k+1}-x^*)^T\right] + o\left(k^{-\frac{1+q}2}\sqrt{E\left[\left\|\bar{x}^{lin}_{k+1}-x^*\right\|^2\right]}\right) + o(k^{-(1+q)}).
\label{eq:covasymp}
\end{equation}

Let now, for some high enough $k_0$ such that $||a_{k_0}DS||_2 < 1$,
\lqq F_{k+1} = (I - a_{k_0}DS)^{-1}(I - a_{k_0+1}DS)^{-1}...(I - a_kDS)^{-1}. \rqq

We may then write, for $k\geq k_0$:
\lqq F_{k+1}(x^{lin}_{k+1}-x^*) = F_k(x^{lin}_k-x^*) - a_kF_{k+1}D\mathcal{E}_k \Rightarrow \rqq
\lqq x^{lin}_{k+1}-x^* = F_{k+1}^{-1}F_{k_0}(x^{lin}_{k_0}-x^*) - \sum_{i=k_0}^{k}a_iF_{k+1}^{-1}F_{i+1}D\mathcal{E}_i \Rightarrow\rqq
\lqq \bar x^{lin}_{k+1}-x^* = O\left(1/\sum_{i=1}^{k}N_i\right) - \frac{\sum_{j=k_0}^k\sum_{i=k_0}^{j}N_ja_iF_{j+1}^{-1}F_{i+1}D\mathcal{E}_i}{\sum_{i=1}^k N_i}. \rqq

Thus,
\lqq \lim_{k\rightarrow\infty}\left(\sum_{j=1}^{k}N_j\right)E\left[(\bar{x}^{lin}_{k+1}-x^*)(\bar{x}^{lin}_{k+1}-x^*)^T\right] =  \rqq
\lqq \lim_{k\rightarrow\infty} \frac{\sum_{i=k_0}^{k}a_i^2\left(\sum_{j=i}^kN_jF_{j+1}^{-1}\right)F_{i+1}DE[\mathcal{E}_i\mathcal{E}_i^T]DF_{i+1}^T\left(\sum_{j=i}^kN_jF_{j+1}^{-1}\right)^T}{\sum_{i=1}^k N_i} = \rqq
\lqq \lim_{k\rightarrow\infty} \frac{q+1}{k^{q+1}}\int_{k_0}^{k}i^{-2\alpha}\left(\int_i^kj^q\Phi(i,j)dj\right)\frac{D\Sigma^2D}{i^q}\left(\int_i^kj^q\Phi(i,j)dj\right)^Tdi\rqq
where $\Phi(i,j) = e^{-DS\frac{j^{1-\alpha}-i^{1-\alpha}}{1-\alpha}}$. Using then that
\lqq\int_i^kj^q\Phi(i,j)dj = \left(\left(i^{q+\alpha}+o(i^{q+\alpha})\right)I - \left(k^{q+\alpha}+o(k^{q+\alpha})\right)\Phi(i,k)\right)(DS)^{-1},\rqq
we obtain:
\lqq \lim_{k\rightarrow\infty}\left(\sum_{j=1}^{k}N_j\right)E\left[(\bar{x}^{lin}_{k+1}-x^*)(\bar{x}^{lin}_{k+1}-x^*)^T\right] =  \rqq
\lqq \lim_{k\rightarrow\infty} \frac{q+1}{k^{q+1}}\int_{k_0}^{k}i^{-2\alpha}\left(i^{q+\alpha}I - k^{q+\alpha}\Phi(i,k)\right)\frac{S^{-1}\Sigma^2S^{-1}}{i^q}\left(i^{q+\alpha}I - k^{q+\alpha}\Phi(i,k)\right)^Tdi =\rqq
\lqq \lim_{k\rightarrow\infty} \frac{q+1}{k^{q+1}}\int_{k_0}^{k}i^{-2\alpha}i^{q+\alpha}\frac{S^{-1}\Sigma^2S^{-1}}{i^q}i^{q+\alpha}di = S^{-1}\Sigma^2S^{-1}. \rqq

Thus, $E\left[\left\|\bar{x}^{lin}_{k+1}-x^*\right\|^2\right] = O(k^{-(1+q)})$ and by Equation~\ref{eq:covasymp}, we have
\lqq E\left[(\bar{x}_{k+1}-x^*)(\bar{x}_{k+1}-x^*)^T\right] = E\left[(\bar{x}^{lin}_{k+1}-x^*)(\bar{x}^{lin}_{k+1}-x^*)^T\right] + o(k^{-(1+q)}) \rqq
\lqq \sim \frac{S^{-1}\Sigma^2S^{-1}}{\sum_{j=1}^{k}N_j}. \rqq
\end{proof}
 
\begin{remark}
The constraint $M \geq 2\lfloor \frac1{\alpha-(1-q)/2} \rfloor+2$ on Theorem~\ref{thm:ipa} is probably unnecessary, if Lemmas 5.2 and 5.3 in Ruppert's work~\cite{avg2} can be generalized to the IP framework. In this case, we would be able to replace $o(k^h)$ with $O(\log k)$ throughout the proof of Theorem~\ref{thm:ipa}, turning the constraint unnecessary.
\end{remark}

\subsection{Asymptotic analysis of the IP-SGD hybrid}
\label{sec:asym-hybrid}
Here we analyze the performance of the IP-SGD hybrid method in a very simplified case where $\uhat J(x) = J$ (constant with respect to $x$, and deterministic), $S=J^TJ$ and $\uhat Q(x_t) = Jy_t + \mathcal{E}_t$, where $y_t = x_t - x^*$, and $\text{Var}[\mathcal{E}_t] = \Sigma^2$, with $E[\mathcal{E}_t| y_1, ..., y_t,\mathcal{E}_1, ..., \mathcal{E}_{t-1}]=0$. Note that under these conditions both IP and SGD would have the same performance, as $\Sigma_B^2 = 0$. Therefore, this analysis cannot prove that the IP-SGD hybrid method performs better than SGD, but only provide sufficient conditions to when it is not \textit{worse}. Particularly, it cannot evince the fact that performance worsens when $\zeta_t$ is asymptotically constant, as was seen in the experiments of Section~\ref{sec:hybrid-exp}, since this case is somewhat equivalent to ``constant'' precision as in SGD.

Assume $A_t = \eta S^{-1}/t$, for some $\eta > 0$. The hybrid method in this case is described by the recurrence:
\lqq \left[\begin{matrix} y_{t+1} \\ \mu_{t+1} \end{matrix}\right] = (I-P_t) \left[\begin{matrix} y_{t} \\ \mu_{t} \end{matrix}\right] + R_t S^{-1}J^T\mathcal{E}_t, \rqq
\lqq P_t=\left[\begin{matrix} \frac{\eta}{2t}I & \frac{\eta}{2t}I \\ -\zeta_tI & \zeta_tI \end{matrix}\right], \text{ } R_t = \left[\begin{matrix} -\frac{\eta}{2t}I \\ \zeta_tI\end{matrix}\right], \rqq
where $\mu_t \triangleq S^{-1}J^T\bar Q_t$.

Let then $W_t = E\left[\left[\begin{matrix} y_{t} \\ \mu_{t} \end{matrix}\right]\left[\begin{matrix} y_{t} \\ \mu_{t} \end{matrix}\right]^T\right]$. $W_t$ can be calculated by the recurrence:
\lqq W_{t+1} = (I-P_t)W_t(I-P_t) + R_tS^{-1}\Sigma_A^2S^{-1} R_t^T, \rqq
noting that $\Sigma_A^2 = J^T\Sigma^2J$.

This recurrence takes different behaviors depending on $\zeta_t$. If $\zeta_t \sim (1+p)/t$, for $p\geq0$, then $E[y_ty_t^T]\sim \frac{a}{t}S^{-1}\Sigma_A^2S^{-1}$, where
\lqq a = e_1^TU\left((U^{-1}rr^TU^{-T})\circ\left[\begin{matrix} \frac{1}{2\lambda_1 - 1} & \frac{1}{\lambda_1 + \lambda_2 - 1} \\ \frac{1}{\lambda_1 + \lambda_2 - 1} & \frac{1}{2\lambda_2 - 1} \end{matrix}\right]\right)U^{T}e_1, \rqq
where
\lqq U\left[\begin{matrix} \lambda_1 & 0 \\ 0 & \lambda_2 \end{matrix}\right]U^{-1} = \left[\begin{matrix} \eta/2 & \eta/2 \\ -(p+1) & p+1 \end{matrix}\right], \text{ } r = \left[\begin{matrix} -\eta/2 \\ p+1 \end{matrix}\right], \rqq
and ``$\circ$'' indicates componentwise product, assuming $\min\{\operatorname{Re}(\lambda_1), \operatorname{Re}(\lambda_2)\} > \frac12$. The solution to $a$ is:
\lqq a(\eta,p) = \frac{\eta^2(16p^2+(22+4\eta)p+8+3\eta)}{(32\eta-16)p^2+(16\eta^2+16\eta-8)p+12\eta^2-4}, \rqq
where the constraint $\min\{\operatorname{Re}(\lambda_1), \operatorname{Re}(\lambda_2)\} > \frac12$ implies that the denominator above must be positive\footnote{A sufficient condition is $\eta>\frac12$, $p \geq 0$.}. We may verify that under these constraints, $a(\eta,p)$ is minimized as $p\rightarrow\infty$ and $\eta\rightarrow1$, satisfying $\min_{\eta} a(\eta,p) > 1$ for all $p$ and $\lim_{p\rightarrow\infty}\arg\min_{\eta} a(\eta,p) = 1$, with:
\lqq \min_{\eta} a(\eta,p) \approx a(1,p) \approx 1 + \frac{1/8}{p+1/3}. \rqq

\subsubsection{Averaged case}
When averaging is considered, the performance analysis changes. Let now $A_t=\eta S^{-1}/t^\alpha$, $\frac12<\alpha<1$. The recurrence now writes as:
\lqq \left[\begin{matrix} \tilde y_{t+1} \\ \mu_{t+1} \\ y_{t} \end{matrix}\right] = (I-P_t) \left[\begin{matrix} \tilde y_{t} \\ \mu_{t} \\ y_{t-1} \end{matrix}\right] + R_t S^{-1}J^T\mathcal{E}_t, \rqq
\lqq P_t=\left[\begin{matrix} \frac{\eta}{2t^\alpha}I & \frac{\eta}{2t^\alpha}I & 0 \\ -\zeta_tI & \zeta_tI & 0 \\ -\frac I{t-1} & 0 & \frac I{t-1} \end{matrix}\right], \text{ } R_t = \left[\begin{matrix} -\frac{\eta}{2t^\alpha}I \\ \zeta_tI \\ 0\end{matrix}\right]. \rqq
where $\tilde y_t = \tilde x_t - x^*$. Once again, we write $W_t = E\left[\left[\begin{matrix} \tilde y_{t} \\ \mu_{t} \\ y_{t-1} \end{matrix}\right]\left[\begin{matrix} \tilde y_{t} \\ \mu_{t} \\ y_{t-1} \end{matrix}\right]^T\right]$, and obtain the recurrence:
\lqq W_{t+1} = (I-P_t)W_t(I-P_t) + R_tS^{-1}\Sigma_A^2S^{-1} R_t^T. \rqq

While this time we did not solve this equation analytically, we verified numerically that when $\zeta_t\sim(1+p)/t$, $p\geq 0$, it appears that $E[y_ty_t^T]\sim \frac{1}{t}S^{-1}\Sigma_A^2S^{-1}$ as $t\rightarrow\infty$, regardless of the choice of $p$ or $\eta$.

\subsection{Stability of Gauss-Newton estimators}
\label{sec:stability}
The difference in stability of different Gauss-Newton estimators can be explained in terms of the number of finite moments of the distribution of $B_t^{-1}$, as convergence requires a minimum number of finite moments. 

\begin{lemma}
Let $A\in\mathbb{R}^{m\times n}$ be a random matrix with $m\geq n$, satisfying $P[||A||_F > C] = 0$ for some $C$, and $\sup_{A\in\mathbb{R}^{m\times n}} \text{pdf}[A] < +\infty$. Then $(A^TA)^{-1}$ has a finite $p$-th moment for $p < \frac{m-n+1}2$.
\end{lemma}
\begin{proof}
We can reduce the problem of verifying if the moments of $(A^TA)^{-1}$ are finite to the problem of verifying if the moments of $||(A^TA)^{-1}||_2$ are finite, where $||.||_2$ is the induced L2 norm for matrices.  Note that $||(A^TA)^{-1}||_2 = \sigma_{\min}^{-2}$, where $\sigma_{\min} := \sigma_{\min}(A)$ is the smallest singular value of $A$. The $p$-th moment of this expression is then:
\lqq E\left[||(A^TA)^{-1}||_2^p\right] = \int_0^\infty \sigma_{\min}^{-2p} \text{pdf}[\sigma_{\min}] d\sigma_{\min} \rqq
\lqq \leq \frac1{r^{2p}} + \int_0^r s^{-2p} \frac d{ds}P[\sigma_{\min} < s] ds \text{ } \text{ } \text{ }(\forall r>0) \rqq

Now, $\sigma_{\min}$ may also be defined as the distance between $A$ and its closest matrix $\tilde A$ such that $\det(\tilde A^T\tilde A)=0$, i.e.:
\lqq \sigma_{\min} = \min_{\substack{ \text{ }\text{ }\text{ }\tilde A \in \mathbb{R}^{m\times n} \\ \text{s.t.: } \det(\tilde A^T\tilde A)=0 }} ||A-\tilde A||_F. \rqq
Since $\{\tilde A \in \mathbb{R}^{m\times n} | \det(\tilde A^T\tilde A)=0 \}$ is an algebraic variety of dimension $mn-\alpha$, where $\alpha = m-n+1$, we know that the hyper-volume (i.e. the Lebesgue measure in $\mathbb{R}^{m\times n}$) of the set $\left\{A\in\mathbb{R}^{m\times n}\text{ }|\text{ }\left\|A\right\|_F<C,\sigma_{\min}(A) < s\right\}$ approaches zero with rate $O(s^\alpha)$ as $s\rightarrow0$, and by boundedness of $\text{pdf}[A]$ this implies that $P[\sigma_{\min} < s] = O(s^\alpha)$.

Thus,
\lqq E\left[||(A^TA)^{-1}||_2^p\right] \leq \frac1{r^{2p}} + \int_0^r s^{-2p} O\left(s^{m-n}\right) ds, \rqq
which converges for $m-n-2p+1 > 0$.
\end{proof}

The lemma above implies that, if $\uhat J$ and its pdf are bounded, then using $B^{-1} = G^TG$, where $G=\sum_{i=1}^Na_i \uhat J_i$, for some sequence $a_i$, we may only guarantee a finite $p$-th moment for $p < \frac{m-n+1}2$, regardless of $N$.

However, if one uses $B^{-1} = R^{-1}$, where $R = \frac1N\sum_{i=1}^N \uhat J_i^T\uhat J_i$, we may write $R=A^TA$, where $A=\frac1{\sqrt{N}}\left[\begin{smallmatrix}\uhat J_1 \\ ... \\ \uhat J_N \end{smallmatrix}\right] \in \mathbb{R}^{Nm\times n}$, and therefore a finite $p$-th moment can be guaranteed for $p < \frac{mN-n+1}2$. This is also valid for any linear combination of the form $B^{-1} = (aG^TG + bR)^{-1}$, with $a,b>0$, which yields a finite $p$-th moment for $N > \frac{2p+n-1}{m}$. This suggests that, when employing Equation~\ref{eq:sgn} to compute $B_t$, there should exist some $t$ after which we would have enough finite moments for convergence.

\section{Conclusion}

This work presented stochastic optimization schemes targeted at MCLS problems. We first introduced the concept of increasing precision (IP), which asymptotically outperforms the other ``constant'' precision methods such as SGD, and then proposed a hybrid approach that substantially improves pre-asymptotic performance. Finally, we also showed how the use of Gauss-Newton can be highly beneficial on MCLS, outperforming Quasi-Newton approaches and covariance-preconditioning methods such as AdaGrad and Adam in our experimental analysis.

It remains as future work to provide a more complete theoretical analysis of the convergence of the hybrid methods, as well as exploring the generalizations of Section~\ref{sec:limitations}, such as providing more generic solutions targeting the cases when the cost assumptions of Section~\ref{sec:assumptions} do not hold, or the case of biased but consistent estimators. Also left for future work is a more extensive comparison of different ways to compute a Gauss-Newton estimator $B_t$ rather than Equation~\ref{eq:sgn}.

\appendix

\section{Definitions of the problems used in experiments}
\label{sec:problems}
\begin{itemize}
\item \textbf{Problem \#1}: The problem dimensions $m,n$ (as in $Q:\mathbb{R}^n\rightarrow\mathbb{R}^m$) are $m=3$, $n=2$, with $\uhat Q(x) = \left[\begin{smallmatrix} L(y) -0.5 \\ L(y-1) -0.5 \\ L(2y-1) -0.2 \end{smallmatrix}\right]$, where $y = x_1 + x_2G$, $G$ is a Gaussian variable with zero mean and unit variance, and $L(y) = \frac{|y| - |y-1| + 1}2$. We use $\uhat J(x) = \left[\begin{smallmatrix} L^\prime(y) \\ L^\prime(y-1) \\ 2L^\prime(2y-1) \end{smallmatrix}\right]\left[\begin{smallmatrix} 1 \\ G \end{smallmatrix}\right]^T$, where ``$\text{ }^\prime$'' indicates derivative. Constraints: $x_1 \in [0,2]$, $x_2 \in [1,3]$. $x^*$ and $S$ were computed numerically for this problem.
\item \textbf{Problem \#2}: The problem dimensions $m,n$ are arbitrary. $\uhat Q(x) = ABu-y$, where $y_i = i^2$, $A\in\mathbb{R}^{m\times n}$ satisfying $A_{ij} = e^{-\frac12\left(nm(i/m-j/n)^2\right)}$, $B\in\mathbb{R}^{n\times n}$ is a diagonal matrix whose entries are i.i.d. random variables uniformly distributed in $[0.15,0.85]$, and $u_i = \sinh(x_i)$, $u,x\in\mathbb{R}^n$. We use $\uhat J(x) = ABu^\prime$, where ``$\text{ }^\prime$'' indicates derivative with respect to $x$.
We constrain $x_i \in [\sinh^{-1}(u^*_i-1), \sinh^{-1}(u^*_i+1)]$, where $u^* = 2(A^TA)^{-1}A^Ty$.
\end{itemize}

\subsection*{Initial value} For Problem \#1, the initial value is always $x_1=(2,1)^T$, for all methods. For Problem \#2, $x_1$ is randomly selected, but different methods in a same figure use all the same initial value.

\subsection*{Enforcing constraints}
Note that the problems used in our experimental analysis are all constrained, although the methods we evaluate in this work were all designed for unconstrained problems. In our implementation, we enforce constraints by projecting an iterate $x_k$ back to the closest feasible point whenever it violates the problem constraints.

\subsection*{Numerical solution of Problem \#1}

For Problem \#1, $x^*$ was computed numerically. We ran 8 parallel instances of the averaged SGD algorithm with $\alpha=.66$, and $N=10$ samples per iteration, for $4\times10^6$ iterations, starting at $x_1=(1,1)^T$, and preconditioned with $D=(\widehat{J^TJ})^{-1}$, where $\widehat{J^TJ}$ is an unbiased estimator of $J^TJ$ at $x_1$, computed with 1000 samples.

The average output of the 8 instances was of $x^* = \smallmat{0.660877 \\ 2.28548}$, and the variance matrix of the mean (of the 8 instances) was $\smallmat{3.72107\times10^{-8} & -7.943\times10^{-8} \\ -7.943\times10^{-8} & 3.12237\times10^{-7}}$.

\section{Variance of gradient estimators}

\begin{lemma}
\label{lemma:var1}
Let $(\uhat Q, \uhat J)$ and $(\uhat Q^\prime, \uhat J^\prime)$ be two i.i.d. pairs of unbiased estimators of a vector $Q$ and matrix $J$ (i.e. with $(\uhat Q, \uhat J) \pperp (\uhat Q^\prime, \uhat J^\prime)$). Then the variance of the estimator $\uhat J^T\uhat Q^\prime$ is greater than or equal to that of $\frac{\uhat J^T\uhat Q^\prime + \left.\uhat J^\prime\right.^T\uhat Q}2$.
\end{lemma}
\begin{proof}
\lqq \text{Var}[\uhat J^T\uhat Q^\prime] - \text{Var}\left[\frac{\uhat J^T\uhat Q^\prime + \left.\uhat J^\prime\right.^T\uhat Q}2\right] =  \rqq
\lqq E\left[\left(\uhat J^T\uhat Q^\prime\right)\left(\uhat J^T\uhat Q^\prime\right)^T\right] - E\left[\left(\frac{\uhat J^T\uhat Q^\prime + \left.\uhat J^\prime\right.^T\uhat Q}2\right)\left(\frac{\uhat J^T\uhat Q^\prime + \left.\uhat J^\prime\right.^T\uhat Q}2\right)^T\right] =  \rqq
\lqq E\left[\frac{\uhat J^T\uhat Q^\prime\left.\uhat Q^\prime\right.^T\uhat J + \left.\uhat J^\prime\right.^T\uhat Q\uhat Q^T\uhat J^\prime}2\right] - E\left[\left(\frac{\uhat J^T\uhat Q^\prime + \left.\uhat J^\prime\right.^T\uhat Q}2\right)\left(\frac{\uhat J^T\uhat Q^\prime + \left.\uhat J^\prime\right.^T\uhat Q}2\right)^T\right] =  \rqq
\lqq E\left[\left(\frac{\uhat J^T\uhat Q^\prime - \left.\uhat J^\prime\right.^T\uhat Q}2\right)\left(\frac{\uhat J^T\uhat Q^\prime - \left.\uhat J^\prime\right.^T\uhat Q}2\right)^T\right] \succeq 0. \rqq
\end{proof}

\begin{lemma}
\label{lemma:var2}
Given $N$ i.i.d. pairs $(\uhat Q^{(i)}, \uhat J^{(i)})$, $i\in\{1,...,N\}$, with $E[\uhat Q^{(i)}] = Q$ and $E[\uhat J^{(i)}] = J$, then 
\lqq \text{Var}\left[\frac{\sum_{i\neq j}\left.\uhat J^{(i)}\right.^T\uhat Q^{(j)}}{N(N-1)}\right] = \frac{\Sigma_A^2}{N} + \frac{\Sigma_B^2}{N(N-1)}, \rqq
where $\Sigma_A^2$ and $\Sigma_B^2$ are positive semidefinite matrices.
\end{lemma}
\begin{proof}
Let us first write
\lqq \text{Var}\left[\frac{\sum_{i\neq j}\left.\uhat J^{(i)}\right.^T\uhat Q^{(j)}}{N(N-1)}\right] = \frac{\sum_{i\neq j}\sum_{k\neq l}E\left[\left.\uhat J^{(i)}\right.^T\uhat Q^{(j)}\left.\uhat Q^{(k)}\right.^T\uhat J^{(l)}-J^TQQ^TJ\right]}{N^2(N-1)^2}. \rqq

From now on, we omit the transpose ``$\text{ }^T$'' symbol for simplicity of notation, so the expression above writes as:
\lqq \frac{\sum_{i\neq j}\sum_{k\neq l}E\left[\uhat J^{(i)}\uhat Q^{(j)}\uhat Q^{(k)}\uhat J^{(l)}-JQQJ\right]}{N^2(N-1)^2}. \rqq

Note that the expectation term above is only nonzero when $i\neq j \neq k \neq l$ does not hold. In the summation above, there are four cases where these variables take altogether three distinct values (which is when either $i=k$, $i=l$, $j=k$, or $k=l$), each one occurring $N(N-1)(N-2)$ times, while there are two cases where they take two distinct values (when $i=k$ and $j=l$, or $i=l$ and $j=k$), each one occurring $N(N-1)$ times. Thus, the expression above writes as:
\lqq E\left[\frac{(N-2)\left(J\uhat Q\uhat QJ + J\uhat QQ\uhat J + \uhat JQ\uhat QJ + \uhat JQQ\uhat J - 4JQQJ\right)}{N(N-1)} + ... \right. \rqq
\lqq \left.\frac{\uhat J\uhat Q^\prime\uhat Q^\prime\uhat J + \uhat J\uhat Q^\prime\uhat Q\uhat J^\prime - 2JQQJ}{N(N-1)}\right] = \frac{\Sigma_A^2}{N} + \frac{\Sigma_B^2}{N(N-1)}, \rqq
with
\lqq
\Sigma_A^2 = E\left[J\uhat Q\uhat QJ + J\uhat QQ\uhat J + \uhat JQ\uhat QJ + \uhat JQQ\uhat J - 4JQQJ\right]
= \text{Var}[J\uhat Q + \uhat JQ], 
\rqq
and
\lqq \Sigma_B^2 = E\left[\uhat J\uhat Q^\prime\uhat Q^\prime\uhat J + \uhat J\uhat Q^\prime\uhat Q\uhat J^\prime - J\uhat Q\uhat QJ - J\uhat QQ\uhat J - \uhat JQ\uhat QJ - \uhat JQQ\uhat J + 2JQQJ\right]. \rqq

Let now $\epsilon_{J} = \uhat J-J$, $\epsilon_{Q} = \uhat Q-Q$, $\epsilon_{J}^\prime = \uhat J^\prime-J$, $\epsilon_{Q}^\prime = \uhat Q^\prime-Q$, and note that for two random variables $X,Y$ written in this notation we may write $E[\uhat{X}\uhat{X}] = XX + E[\epsilon_{X}\epsilon_{X}]$ and $E[\uhat{X}\uhat{Y}] = XY + E[\epsilon_{X}\epsilon_{Y}]$.

We can then simplify:
\lqq \Sigma_B^2 = E[\uhat{J}\uhat{Q}^\prime \uhat{Q}^\prime \uhat{J}+\uhat{J}^\prime \uhat{Q}\uhat{Q}^\prime \uhat{J}-J\uhat{Q}\uhat{Q}J-\uhat{J}QQ\uhat{J}-J\uhat{Q}Q\uhat{J}-\uhat{J}Q\uhat{Q}J+2JQQJ] =  \rqq
\lqq E[\uhat{J}(QQ+\epsilon_{Q}^\prime \epsilon_{Q}^\prime )\uhat{J}+\uhat{J}^\prime \uhat{Q}\uhat{Q}^\prime \uhat{J}-J\uhat{Q}\uhat{Q}J-\uhat{J}QQ\uhat{J}-J\uhat{Q}Q\uhat{J}-\uhat{J}Q\uhat{Q}J+2JQQJ] =  \rqq
\lqq E[\uhat{J}\epsilon_{Q}^\prime \epsilon_{Q}^\prime \uhat{J}+\uhat{J}^\prime \uhat{Q}\uhat{Q}^\prime \uhat{J}-J\uhat{Q}\uhat{Q}J-J\uhat{Q}Q\uhat{J}-\uhat{J}Q\uhat{Q}J+2JQQJ] =  \rqq
\lqq E[\uhat{J}\epsilon_{Q}^\prime \epsilon_{Q}^\prime \uhat{J}+\epsilon_{J}^\prime \uhat{Q}\epsilon_{Q}^\prime \uhat{J}-J\epsilon_{Q}\epsilon_{Q}J-\epsilon_{J}Q\epsilon_{Q}J] =  \rqq
\lqq E[\uhat{J}\epsilon_{Q}^\prime \epsilon_{Q}^\prime \uhat{J}+\epsilon_{J}^\prime \uhat{Q}\epsilon_{Q}^\prime \uhat{J}-J\epsilon_{Q}^\prime \epsilon_{Q}^\prime J-\epsilon_{J}^\prime Q\epsilon_{Q}^\prime J] =  \rqq
\lqq E[\epsilon_{J}\epsilon_{Q}^\prime \epsilon_{Q}^\prime \epsilon_{J}+\epsilon_{J}^\prime \epsilon_{Q}\epsilon_{Q}^\prime \epsilon_{J}] =  \rqq
\lqq \frac12E[\epsilon_{J}\epsilon_{Q}^\prime \epsilon_{Q}^\prime \epsilon_{J}+\epsilon_{J}^\prime \epsilon_{Q}\epsilon_{Q}^\prime \epsilon_{J}+\epsilon_{J}^\prime \epsilon_{Q}\epsilon_{Q}\epsilon_{J}^\prime +\epsilon_{J}\epsilon_{Q}^\prime \epsilon_{Q}\epsilon_{J}^\prime ] = \frac12\text{Var}[\epsilon_{J}\epsilon_{Q}^\prime +\epsilon_{J}^\prime \epsilon_{Q}]. \rqq

Since both $\Sigma_A^2$ and $\Sigma_B^2$ can be written as the variance of some expression, they must be positive semidefinite.
\end{proof}

\section{The connection between IP and the hybrid approach}
\label{sec:hybconn}
Starting from the IP update formula (Equations \ref{eq:gradprec} and \ref{eq:ip}), we note that we may write it as:
\begin{align}
x_{k+1} &= x_k - A_k\frac{\sum_{1\leq i\neq j \leq N_k} \left.\uhat J_k^{(i)}\right.^T\uhat Q_k^{(j)}}{N_k(N_k-1)} \nonumber \\
&= x_k - A_k\frac{\sum_{i=1}^{N_k}\left(\left(\sum_{j=1}^{i-1}\uhat J_k^{(j)}\right)^T\uhat Q_k^{(i)} + \left.\uhat J_k^{(i)}\right.^T\left(\sum_{j=1}^{i-1}\uhat Q_k^{(j)}\right)\right)}{N_k(N_k-1)}. \nonumber
\end{align}

Note now that we may break down the equation above in $N_k$ smaller updates of $x$, by defining $\tilde x_{k,1} = x_k$, $\tilde x_{k,N_k+1} = x_{k+1}$, and
\eq{\tilde x_{k,i+1} = \tilde x_{k,i} - A_k\frac{\left(\sum_{j=1}^{i-1}\uhat J_k^{(j)}\right)^T\uhat Q_k^{(i)} + \left.\uhat J_k^{(i)}\right.^T\left(\sum_{j=1}^{i-1}\uhat Q_k^{(j)}\right)}{N_k(N_k-1)}.\label{eq:hybrid2}}

The IP-SGD hybrid method then comes from two modifications of Equation~\ref{eq:hybrid2}. First, we compute the $J_k^{(i)},Q_k^{(i)}$ in function of $\tilde x_{k,i}$ instead of $\tilde x_{k,1}$. Secondly, instead of estimating the gradient in function of the previous $i-1$ samples, we use all previous $t-1$ (with $t = i+\sum_{l=1}^{k-1}N_l$) samples, however giving a higher weight to more recent samples, according to a predefined increasing sequence of positive numbers $q_t$. Renaming now $x_t := \tilde x_{k,i}$, the IP-SGD hybrid method (Equations \ref{eq:ip-sgd-hybrid} and \ref{eq:hybrid1}) is obtained.

\bibliographystyle{siamplain}
\bibliography{article}

\end{document}